\documentclass[
	english
]{siamart251216}
\usepackage[T1]{fontenc}
\usepackage[utf8]{inputenc}
\usepackage{amsfonts,amssymb}
\usepackage{booktabs}

\usepackage{preamble}

\usepackage{pifont}% http://ctan.org/pkg/pifont
\newcommand{\cmark}{\ding{51}}%
\newcommand{\xmark}{\ding{55}}%

\usepackage{enumitem}
\setenumerate{label=\textup{(\roman*)}}

\usepackage{lmodern}
\usepackage[final]{microtype}

\usepackage{tikz}

\usetikzlibrary{arrows}

\tikzset{
    punkt/.style={
           rectangle,
           draw=white, very thick,
           text width=4.5em,
           minimum height=1.5em,
           text centered}
}

\usepackage{mathtools}

\providecommand\given{\nonscript\;\delimsize|\nonscript\;\mathopen{}}
\DeclarePairedDelimiterX\set[1]\{\}{#1}

\DeclarePairedDelimiterX\dual[2]{\langle}{\rangle}{#1,#2}

\DeclarePairedDelimiter\abs{\lvert}{\rvert}
\DeclarePairedDelimiter\norm{\lVert}{\rVert}

\DeclarePairedDelimiter\parens()

\newcommand\N{\mathbb{N}}
\newcommand\R{\mathbb{R}}

\newcommand\LL{\mathcal{L}}

\renewcommand\d{\mathop{}\!\mathrm{d}}

\newcommand{\dualspace}{^\star}
\newcommand{\adjoint}{^\star}
\newcommand{\polar}{^\circ}
\newcommand{\anni}{^\perp}
\newcommand{\tto}{\rightrightarrows}

\newcommand{\cl}{\operatorname{cl}}
\newcommand{\cone}{\operatorname{cone}}
\newcommand{\interior}{\operatorname{int}}
\newcommand{\wsclosure}{\operatorname{cl}_\star}

\newcommand\borel{\mathcal{B}}

\newcommand\radialcone{\mathcal{R}}
\newcommand\tangentcone{\mathcal{T}}
\newcommand\normalcone{\mathcal{N}}

\newcommand{\zerofunction}{\mathbf 0}
\newcommand{\onefunction}{\mathbf 1}

\renewcommand\orcid[1]{%
	\hspace{.25em}%
	\href{http://orcid.org/#1}{%
		\protect\includegraphics[height=1em]{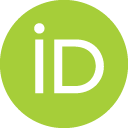}%
	}%
	\hspace{.25em}%
}

\newcommand\mynewsiamremark[3]{
	\newsiamremark{#1}{#2}
	\crefname{#1}{#2}{#3}
	\AddToHook{env/#1/begin}{\crefalias{theorem}{#1}}
}
\mynewsiamremark{example}{Example}{Examples}
\mynewsiamremark{remark}{Remark}{Remarks}

\patchcmd{\SetTagPlusEndMark}${}{}{}
\patchcmd{\SetTagPlusEndMark}${}{}{}

\makeatletter
\def\fullwidthdisplay{\displayindent\z@ \displaywidth\columnwidth}
\edef\@tempa{\noexpand\fullwidthdisplay\the\everydisplay}
\everydisplay\expandafter{\@tempa}
\makeatother

\title{%
	Uniqueness and stability of Lagrange multipliers
	and associated qualification conditions
	}%
\author{%
	Patrick Mehlitz\thanks{%
		Marburg University,
		Department of Mathematics and Computer Science,
		35032 Marburg,
		Germany,
		\email{mehlitz@uni-marburg.de},
		\orcid{0000-0002-9355-850X}%
	}% 
	\and 
	Gerd Wachsmuth\thanks{%
		BTU Cottbus-Senftenberg, 
		Institute of Mathematics, 
		03046 Cottbus, 
		Germany, 
		\email{gerd.wachsmuth@b-tu.de},
		\orcid{0000-0002-3098-1503}%
	}%
	}

	\headers{Patrick Mehlitz, Gerd Wachsmuth}{Uniqueness and stability of Lagrange multipliers}

\begin{document}
%%fakesection: Title
\maketitle

\begin{abstract}
	This paper is concerned with uniqueness and stability
	of Lagrange multipliers for constrained optimization problems in abstract spaces.
	It is well known that validity of the strict Robinson--Zowe--Kurcyusz condition
	implies the so-called isolated calmness,
	a one-sided Lipschitz property tailored for set-valued mappings, 
	of some Lagrange multiplier mapping
	associated with a perturbed version of the original optimization problem,
	and the latter indeed is enough to guarantee uniqueness of the Lagrange multiplier.
	The paper studies the isolated calmness of the Lagrange multiplier mapping in detail.
	Exemplary, it is shown that this condition is sufficient for the Robinson--Zowe--Kurcyusz constraint
	qualification and, in the presence of additional assumptions, 
	even equivalent to the strict
	Robinson--Zowe--Kurcyusz condition.
	Illustrative examples are presented to underline the necessity of postulated assumptions. 
\end{abstract}



\section{Introduction}

We are concerned with the constrained optimization problem
\begin{equation}\label{eq:problem}\tag{P}
	\min\limits_x\quad f(x)\quad\text{s.t.}\quad g(x)\in K
\end{equation}
for continuously Fr\'{e}chet differentiable mappings 
$f\colon X\to\R$ and $g\colon X\to Y$ between Banach spaces $X$ and $Y$
as well as a closed convex set $K\subset Y$.
These assumptions are standing throughout the paper.
Associated with \eqref{eq:problem} is the Lagrangian
$\LL\colon X\times Y\dualspace\to\R$ defined by
\[
	\forall x\in X,\,\forall \lambda\in Y\dualspace\colon\quad
	\LL(x,\lambda)
	\coloneqq 
	f(x) + \dual{\lambda}{g(x)}_Y.
\]
The Karush--Kuhn--Tucker (KKT for brevity) system 
of \eqref{eq:problem} is given by
\begin{equation}\label{eq:KKT}
	\LL'_x(x,\lambda)=0,\quad \lambda\in\normalcone_K(g(x)),
\end{equation}
where $\normalcone_K(g(x))$ is the normal cone to $K$ at $g(x)$.
We emphasize that $\lambda\in\normalcone_K(g(x))$
implicitly requires $g(x)\in K$, 
i.e., feasibility of $x$ for \eqref{eq:problem}.
A primal-dual solution $(x,\lambda)\in X\times Y\dualspace$ of \eqref{eq:KKT}
will be referred to as a KKT pair subsequently,
and $\lambda$ is called a Lagrange multiplier associated with $x$.
Given $x\in X$, the associated set of Lagrange multipliers is
\[
	\Lambda(x)
	\coloneqq
	\set{\lambda\in Y\dualspace\given\LL'_x(x,\lambda)=0,\,\lambda\in\normalcone_K(g(x))}.
\]

In this paper, 
we are concerned with uniqueness and stability of 
Lagrange multipliers associated with \eqref{eq:problem}.
For our study, we consider the perturbed optimization problem
\begin{equation*}\label{eq:pertubed_problem}\tag{P$(\xi,\upsilon)$}
	\min\limits_x\quad f(x)-\dual{\xi}{x}_X
		\quad\text{s.t.}\quad g(x)+\upsilon\in K
\end{equation*}
for parameters $\xi\in X\dualspace$ and $\upsilon\in Y$.
If these parameters vanish, \eqref{eq:problem} is recovered.
The KKT system of \eqref{eq:pertubed_problem} reads as
\[
	\LL'_x(x,\lambda)=\xi,\quad \lambda\in\normalcone_K(g(x)+\upsilon).
\]
Let us fix a KKT pair $(\bar x,\bar\lambda)\in X\times Y\dualspace$ of \eqref{eq:problem}.
In order to study the stability of the Lagrange multiplier $\bar\lambda$ under perturbations,
we introduce the restricted Lagrange multiplier mapping
$\Upsilon_{\bar x}\colon X\dualspace\times Y\tto Y\dualspace$ of \eqref{eq:pertubed_problem},
the main workhorse of this paper,
by means of
\[
	\forall\xi\in X\dualspace,\,\forall \upsilon\in Y\colon\quad
	\Upsilon_{\bar x}(\xi,\upsilon)
	\coloneqq
	\set{
		\lambda\in Y\dualspace
		\given
		\LL'_x(\bar x,\lambda)=\xi,\,\lambda\in\normalcone_K(g(\bar x)+\upsilon)
		}.
\]
Observe that $\bar\lambda\in\Upsilon_{\bar x}(0,0)=\Lambda(\bar x)$.
Our investigations evolve from the so-called isolated calmness
of $\Upsilon_{\bar x}$ at the point $((0,0),\bar\lambda)$ from its graph
which requires the existence of constants $\varepsilon,\delta,c>0$ such that
\begin{equation}\label{eq:isol_calmness_ups}\tag{IC}
	\begin{aligned}
	\forall (\xi,\upsilon)\in B_\varepsilon^{X\dualspace\times Y}((0,0)),\,
	&\forall \lambda\in\Upsilon_{\bar x}(\xi,\upsilon)
		\cap 
		B_\delta^{Y\dualspace}(\bar\lambda)
	\colon\quad
	\\
	&
	\norm{\lambda-\bar \lambda}_{Y\dualspace}
	\leq
	c(\norm{\xi}_{X\dualspace}+\norm{\upsilon}_{Y})
	\end{aligned}
\end{equation}
holds.
Note that \eqref{eq:isol_calmness_ups} implies $\Lambda(\bar x)=\Upsilon_{\bar x}(0,0)=\set{\bar\lambda}$
as $\Lambda(\bar x)$ is a convex set by definition.
Condition \eqref{eq:isol_calmness_ups} further claims that,
for perturbations $(\xi,\upsilon)\in X\dualspace\times Y$ small enough in norm
and each multiplier $\lambda\in Y\dualspace$ certificating stationarity of $\bar x$
for the perturbed problem \eqref{eq:pertubed_problem} and being sufficiently close to $\bar\lambda$,
the distance between $\lambda$ and $\bar\lambda$ is upper bounded by a constant multiple
of the norm of the perturbations, i.e., the restricted multiplier mapping $\Upsilon_{\bar x}$ enjoys
a one-sided local Lipschitz property which is referred to as calmness in the literature.
Following the comments after \cite[Theorem~3I.3]{DontchevRockafellar2014} indicates
that, whenever \eqref{eq:isol_calmness_ups} is valid, then one can even choose $\varepsilon=\infty$ 
potentially after making $\delta$ smaller.

Here, we embed condition \eqref{eq:isol_calmness_ups}
into the landscape of qualification conditions that apply to \eqref{eq:problem}.
As our first main result, we prove in \cref{thm:IC_gives_RZKCQ} that \eqref{eq:isol_calmness_ups}
is sufficient for the validity of the so-called 
Robinson--Zowe--Kurcyusz constraint qualification, see \cite{Robinson1976,ZoweKurcyusz1979},
which reads 
\begin{equation}\label{eq:RZKCQ}\tag{RZKCQ}
	Y = g'(\bar x)X - \radialcone_K(g(\bar x)).
\end{equation}
Let us note that \cite[Theorem~4.1]{ZoweKurcyusz1979} shows that
whenever $\bar x\in X$ is a local minimizer of \eqref{eq:problem}
where \eqref{eq:RZKCQ} holds, then $\Lambda(\bar x)$ is nonempty and bounded.
According to \cite[Proposition~4.47]{BonnansShapiro2000},
\eqref{eq:isol_calmness_ups} (for all $\delta>0$)
is implied by the so-called strict Robinson--Zowe--Kurcyusz condition
\begin{equation}\label{eq:sRZKC}\tag{sRZKC}
	Y
	=
	g'(\bar x) X - \radialcone_K(g(\bar x))\cap\bar\lambda\anni.
\end{equation}
As the latter depends, via the Lagrange multiplier $\bar\lambda$, 
on the objective function of \eqref{eq:problem},
we do not refer to it as a constraints qualification
but call \eqref{eq:sRZKC} a qualification condition instead.
As \eqref{eq:sRZKC} is sufficient for \eqref{eq:isol_calmness_ups} 
it also yields $\Lambda(\bar x)=\set{\bar\lambda}$.
In \cite{Shapiro1997}, further necessary and sufficient conditions for the uniqueness
of multipliers for \eqref{eq:problem} are studied.
Following \cite[Remark~4.49]{BonnansShapiro2000},
\eqref{eq:sRZKC} is equivalent to the so-called 
strict Mangasarian--Fromovitz condition in standard nonlinear optimization
with inequality and equality constraints,
and provides an equivalent characterization of the uniqueness of $\bar\lambda$
in this situation, see \cite{Wachsmuth2013}.
Here, we are addressing the question whether \eqref{eq:isol_calmness_ups}
is sufficient for the validity of \eqref{eq:sRZKC} in more general scenarios.
In \cite[Proposition~4.47]{BonnansShapiro2000},
a partial answer was given merely under fairly restrictive assumptions.
Our second main result, \cref{thm:IC_gives_sRZKCQ_findim},
shows that \eqref{eq:isol_calmness_ups} implies \eqref{eq:sRZKC}
whenever $Y$ is finite dimensional and $K$ is so-called polyhedric.
We also provide some counterexamples which illustrate that none of these
two conditions can be dropped in general.

The remainder of this paper is organized as follows.
In \cref{sec:notation}, 
we comment on the notation used throughout the paper.
\Cref{sec:open_mapping_theorem} investigates an extension
of the well-known generalized open mapping theorem which 
will be used in our analysis but might be of independent interest as well.
Some essentials about qualification conditions addressing problem \eqref{eq:problem}
as well as a discussion about uniqueness of Lagrange multipliers 
are presented in \cref{sec:QCs}.
The main results of the paper, drafted above, are stated and proven in \cref{sec:main}.
Throughout, illustrative examples accompany the theory and exemplify necessity of assumptions.
\Cref{sec:conclusions} closes the paper by summarizing the results 
and mentioning promising aspects of future research.

\section{Notation}\label{sec:notation}

The notation used in this paper is mainly standard and follows the one used in  \cite{BonnansShapiro2000}.

We use $\R_+$ and $\R_-$ for the nonnegative and nonpositive real numbers, respectively.
Given a Banach space $X$, $\norm{\cdot}_X\colon X\to\R_+$ denotes its norm.
Some norm in the Euclidean space $\R^n$ is represented by $\norm{\cdot}\colon\R^n\to\R_+$
for simplicity of notation.
By
\begin{equation*}
	U^X_r(\bar x) := \set{ x \in X \given \norm{x - \bar x}_X < r },
	\quad
	B^X_r(\bar x) := \set{ x \in X \given \norm{x - \bar x}_X \le r }
\end{equation*}
with $\bar x\in X$ and $r \in [0,\infty]$,
we denote open and closed $r$-balls around $\bar x$, respectively.
Furthermore, we exploit
\[
	\R x\coloneqq\set{\alpha x\in X \given \alpha\in\R},
\]
and $\R_+x$ as well as $\R_-x$ may be defined in analogous fashion. 
The (topological) dual space of $X$ will be represented by $X\dualspace$,
and $\dual{\cdot}{\cdot}_X\colon X\dualspace\times X\to\R$ is used
for the associated dual pairing.
For a set $M\subset X$,
\[
	M\polar 
	:= 
	\set{
		x\dualspace\in X\dualspace 
		\given 
		\forall x\in M\colon\,\dual{x\dualspace}{x}_X\leq 0
	}
\]
is the polar cone of $M$ which is a closed convex cone.
Moreover, $\cl(M)$, $\interior(M)$, and $\cone(M)$ denote 
the closure, the interior, and the conic hull of $M$, respectively.
The function $\iota_M\colon X\to\R\cup\set{\infty}$
given by $\iota_M(x):=0$ if $x\in M$ and $\iota_ M(x):=\infty$
if $x\in X\setminus M$ is the so-called indicator function of $M$.
Assuming that $M$ is closed and convex, and picking $\bar x\in M$,
\[
	\radialcone_M(\bar x)
	:=
	\cone(M-\set{\bar x}),
	\quad
	\tangentcone_M(\bar x)
	:=
	\cl\parens*{\radialcone_M(\bar x)},
	\quad
	\normalcone_M(\bar x)
	:=
	\parens*{\tangentcone_M(\bar x)}\polar
\]
represent the radial cone, the tangent cone, and the normal cone 
to $M$ at $\bar x$, respectively.
We note that $\normalcone_M(\bar x)=(M-\set{\bar x})\polar$ holds,
and whenever $M$ is a cone, we have $\radialcone_M(\bar x)=M+\R_-\bar x$.
Given $\tilde x\notin M$, 
we set $\radialcone_M(\tilde x):=\emptyset$, $\tangentcone_M(\tilde x):=\emptyset$,
and $\normalcone_M(\tilde x):=\emptyset$ for the purpose of completeness.
For $x\dualspace\in X\dualspace$,
\[
	{x\dualspace}\anni
	:=
	\set{
		x\in X
		\given
		\dual{x\dualspace}{x}_X=0
	}
\]
represents the (pre-)annihilator of $x\dualspace$.
Fixing $x\dualspace\in\normalcone_M(\bar x)$,
$M$ is referred to as polyhedric w.r.t.\ $(\bar x,x\dualspace)$
if
\[
	\tangentcone_M(\bar x)\cap{x\dualspace}\anni
	=
	\cl\parens*{\radialcone_M(\bar x)\cap{x\dualspace}\anni}
\]
is valid, and if this holds for all $\bar x\in M$
and all $x\dualspace\in\normalcone_M(\bar x)$,
then $M$ is called polyhedric.
For detailed information about and examples of polyhedric sets,
the interested reader is referred to the survey paper \cite{Wachsmuth2019}.
For $M\dualspace\subset X\dualspace$, 
$\wsclosure M\dualspace$ is the weak-$\star$ closure of $M\dualspace$.

For a convex function $\varphi\colon X\to\R\cup\set{\infty}$
and some point $\bar x\in X$ such that $\varphi(\bar x)<\infty$,
\[
	\partial \varphi(\bar x)
	:=
	\set{
		x\dualspace\in X\dualspace
		\given
		\forall x\in X\colon\,
		\varphi(x)\geq \varphi(\bar x) + \dual{x\dualspace}{x-\bar x}_X
		}
\]
represents the subdifferential of $\varphi$ at $\bar x$.
Whenever $M\subset X$ is a closed convex set and $\bar x\in M$ is fixed,
$\partial\iota_M(\bar x)=\normalcone_M(\bar x)$ holds by definition.

Given yet another Banach space $Y$, 
$L(X,Y)$ denotes the Banach space of all bounded linear operators $A\colon X\to Y$.
Given $A\in L(X,Y)$, $\ker A:=\set{x\in X \given Ax = 0}$ is the kernel of $A$,
and we use $A\adjoint\in L(Y\dualspace,X\dualspace)$ to represent the adjoint of $A$.
Let us emphasize that $(AX)\polar = \ker A\adjoint$ holds by definition of the adjoint.
For a Fr\'{e}chet differentiable mapping $F\colon X\to Y$ and $\bar x\in X$,
$F'(\bar x)\in L(X,Y)$ represents the Fr\'{e}chet derivative of $F$ at $\bar x$.

By
$L^2(0,1)$ ($L^\infty(0,1)$) we denote the usual Lebesgue space of
equivalence classes of
real-valued, measurable, and square-integrable (essentially bounded) functions on $(0,1)$
equipped with the common norm.
The set of all functions in $L^2(0,1)$ which only take
nonnegative values (up to sets of measure zero) 
is denoted by $L^2(0,1)_+$,
and $L^2(0,1)_-$ is defined in similar fashion.
Whenever $I \subset (0,1)$ is measurable, 
$\abs{I}$ is the (Lebesgue) measure of $I$,
$L^2(I) \subset L^2(0,1)$ is the subspace of functions vanishing outside of $I$,
and $\chi_{I}\in L^\infty(0,1)$
denotes the characteristic function of $I$
which takes value $1$ on $I$ and vanishes on $(0,1) \setminus I$.
Whenever $\Theta\subset\R$ is nonempty and compact,
$C(\Theta)$ denotes the Banach space of all real-valued, continuous functions on $\Theta$
which is equipped with the usual supremum norm.
We use $\mathcal B(\Theta)$ to represent the Borelean $\sigma$-algebra induced by $\Theta$.
The topological dual space $\mathcal M(\Theta)$ of $C(\Theta)$ 
is the Banach space of all signed regular measures
on the measurable space $(\Theta,\mathcal B(\Theta))$ and equipped with the standard variation norm.
For clarity, the real-valued functions $\omega\mapsto 0$ and $\omega\mapsto 1$
(on a generic domain) will be denoted by $\zerofunction$ and $\onefunction$, respectively.

\section{Revisiting the generalized open mapping theorem}\label{sec:open_mapping_theorem}

We present some more insights 
into the classical generalized open mapping theorem from \cite[Theorem~2.1]{ZoweKurcyusz1979}.
Although needed later on,
it also might be of independent interest
which is why we present it in a separate section.

\begin{theorem}
	\label{thm:gen_open_2}
	Let $X$ and $Y$ be Banach spaces, fix $A \in L(X,Y)$, and let $C \subset Y$ be a closed convex cone.
	Then the following are equivalent:
	\begin{enumerate}
		\item
			\label{thm:gen_open_2:1}
			$Y = A X - C$,
		\item
			\label{thm:gen_open_2:2}
			there exist $\rho \in (0,\infty)$ and $R \in (0,\infty]$
			s.t.\ 
			$U^Y_\rho(0) \subset A B^X_1(0) - C \cap B^Y_R(0)$,
		\item
			\label{thm:gen_open_2:3}
			there exist $\rho \in (0,\infty)$ and $R \in (0,\infty]$
			s.t.\ 
			$U^Y_\rho(0) \subset \cl\parens*{ A B^X_1(0) - C \cap B^Y_R(0) }$,
		\item
			\label{thm:gen_open_2:4}
			there exists $\rho \in (0,\infty)$
			such that
			\begin{equation}
				\label{eq:boundedness_2}
				\forall y\dualspace \in C\polar\colon\quad
				\norm{y\dualspace}_{Y\dualspace}
				\le
				\frac1{\rho} \norm{A\adjoint y\dualspace}_{X\dualspace}
				.
			\end{equation}
	\end{enumerate}
	If these conditions hold, the same constants $\rho$ and $R$
	can be chosen in
	\ref{thm:gen_open_2:2} and \ref{thm:gen_open_2:3}.
	Conditions \ref{thm:gen_open_2:2} and \ref{thm:gen_open_2:3}
	with constants $\rho$ and $R$
	imply \ref{thm:gen_open_2:4} with the same constant $\rho$.
	Finally, \ref{thm:gen_open_2:4} with $\rho$
	implies
	\ref{thm:gen_open_2:2} and \ref{thm:gen_open_2:3}
	with the same $\rho$ and any $R \in [\rho + \norm{A}_{L(X,Y)}, \infty]$.
\end{theorem}
\begin{proof}
	Equivalence of \ref{thm:gen_open_2:1}, \ref{thm:gen_open_2:2}, and \ref{thm:gen_open_2:3}
	follows from \cite[Theorem~2.1]{ZoweKurcyusz1979} and its proof.

	[\ref{thm:gen_open_2:2}$\Longrightarrow$\ref{thm:gen_open_2:4}]:
	Pick $y\dualspace\in C\polar$ arbitrarily.
	Due to \ref{thm:gen_open_2:2}, for each $z\in U_1^Y(0)$,
	we find $x\in B_1^X(0)$ and $y\in C\cap B_R^Y(0)$ such that $-\rho z = Ax - y$.
	Hence, by the definition of the polar cone, we obtain
	\[
		\dual{y\dualspace}{z}_Y
		=
		\frac1\rho\dual{y\dualspace}{y}_Y
		-
		\frac1\rho\dual{A\adjoint y\dualspace}{x}_X
		\leq
		- 
		\frac{1}{\rho}\dual{A\adjoint y\dualspace}{x}_X
		\le
		\frac{1}{\rho} \norm{A\adjoint y\dualspace}_{X\dualspace}.
	\]
	Since $z \in U_1^Y(0)$ was arbitrary,
	the definition of the dual norm
	implies \eqref{eq:boundedness_2}.

	[\ref{thm:gen_open_2:4}$\Longrightarrow$\ref{thm:gen_open_2:3}]:
	We argue by contradiction.
	Suppose that 
	\ref{thm:gen_open_2:4} holds
	but \ref{thm:gen_open_2:3} is violated
	for the particular constant $\rho>0$ from \ref{thm:gen_open_2:4}
	and for $R := \rho + \norm{A}_{L(X, Y)}$.
	Consequently,
	there exists $\bar y \in  U_\rho^Y(0)$ such that
	\[
		\bar y\notin\cl\parens*{ A B^X_1(0) - C \cap B^Y_R(0) }.
	\]
	As the set on the right is closed and convex,
	applying \cite[Theorem 1.5.9]{Schirotzek2007} equips us 
	with $y\dualspace \in Y\dualspace \setminus \set{0}$
	and $\eta > 0$
	such that
	\begin{equation*}
		\forall x \in B^X_1(0),\,\forall y \in C \cap B^Y_R(0)\colon\quad
		\dual{A\adjoint y\dualspace}{x}_X - \dual{y\dualspace}{y}_Y
		\ge
		\dual{y\dualspace}{\bar y}_Y
		+
		2\eta
		.
	\end{equation*}
	Taking the infimum over all $x\in B^X_1(0)$ yields
	\begin{equation*}
		\forall y \in C \cap B^Y_R(0)\colon\quad
		\dual{-y\dualspace}{y}_Y
		\ge
		\norm{A\adjoint y\dualspace}_{X\dualspace}
		- 
		\rho \norm{y\dualspace}_{Y\dualspace}
		+
		2\eta
		.
	\end{equation*}
	From this inequality
	together with
	$\eta>0$ we find that
	\[
		\varepsilon
		:=
		\rho \norm{y\dualspace}_{Y\dualspace}
		-
		\norm{A\adjoint y\dualspace}_{X\dualspace}
		-
		\eta
		>
		0
		,
	\]
	and $0$ is an $\varepsilon$-minimal point of the function
	$\iota_{C \cap B^Y_R(0)} - y\dualspace$.
	Hence, applying Ekeland's variational principle,
	see \cite[Theorem~1.1]{Ekeland1974},
	for some $\kappa \in (0, R)$, 
	we find $\tilde y \in C \cap B^Y_R(0)$
	with $\norm{\tilde y}_Y \le \kappa$
	such that $\tilde y$ is the uniquely determined minimizer
	of $\iota_{C \cap B^Y_R(0)} - y\dualspace + \varepsilon/\kappa\,\norm{\cdot-\tilde y}_Y$.
	Hence, exploiting some calculus rules for the subdifferential,
	see, e.g., \cite[Remark~4.4.2, Propositions~4.5.1 and~4.6.2]{Schirotzek2007}, we find
	\begin{equation*}
		0 \in \partial\iota_{C \cap B^Y_R(0)}(\tilde y) 
		- \set{y\dualspace} + B^{Y\dualspace}_{\varepsilon / \kappa}(0).
	\end{equation*}
	Note that $\norm{\tilde y}_Y\leq\kappa<R$ implies
	\begin{equation*}
		\partial\iota_{C \cap B^Y_R(0)}(\tilde y)
		=
		\normalcone_{C \cap B^Y_R(0)}(\tilde y)
		=
		\normalcone_{C}(\tilde y)
		\subset
		C\polar
	\end{equation*}
	as $C$ is a closed convex cone.
	Thus, there exists $\tilde y\dualspace \in C\polar$
	with $\norm{y\dualspace - \tilde y\dualspace}_{Y\dualspace} \le \varepsilon / \kappa$.
	From \eqref{eq:boundedness_2}
	and from the choice of $\varepsilon$
	we get
	\begin{align*}
		0
		&\le
		\norm{ A\adjoint \tilde y\dualspace}_{X\dualspace}
		-
		\rho \norm{\tilde y\dualspace}_{Y\dualspace}
		\\
		&\le
		\norm{ A\adjoint y\dualspace}_{X\dualspace}
		-
		\rho \norm{y\dualspace}_{Y\dualspace}
		+
		\parens*{ \norm{A}_{L(X, Y)} + \rho } \norm{y\dualspace - \tilde y\dualspace}_{Y\dualspace}
		\\
		&\le
		-
		(\varepsilon + \eta)
		+
		\parens*{ \norm{A}_{L(X, Y)} + \rho } \frac{\varepsilon}{\kappa}
		.
	\end{align*}
	Note that $\tilde y$ and $\tilde y\dualspace$ depend on $\kappa \in (0,R)$
	but all the quantities on the right-hand side of the last inequality
	are independent of $\kappa$.
	Thus, we can pass to the limit $\kappa \to R = \norm{A}_{L(X, Y)} + \rho$
	and
	we arrive at the contradiction
	$\varepsilon + \eta \le \varepsilon$.
	Hence, \ref{thm:gen_open_2:3} holds with the same constant $\rho$ as used in \ref{thm:gen_open_2:4}
	and with $R = \rho + \norm{A}_{L(X, Y)}$.
	Consequently, it also holds for all larger values of $R$.
\end{proof}
We note that the above proof of the implication
[\ref{thm:gen_open_2:2}$\Longrightarrow$\ref{thm:gen_open_2:4}]
is implicitly contained in the proof of \cite[Theorem~4.1(a)]{ZoweKurcyusz1979}.
We are, however,
not aware of any reference which shows that
\eqref{eq:boundedness_2}
implies any of the other conditions.

\section{Qualification conditions and uniqueness of Lagrange multipliers}\label{sec:QCs}

Let us fix a KKT pair $(\bar x,\bar\lambda)\in X\times Y\dualspace$ of \eqref{eq:problem},
and define a closed convex set $\overline K\subset Y$ by means of
\begin{equation}\label{eq:bar_K}
	\overline K\coloneqq\set{y\in K\given\dual{\bar\lambda}{y-g(\bar x)}_Y=0}.
\end{equation}
A simple calculation reveals that
\begin{equation}\label{eq:radialcone_to_bar_K}
	\forall y\in\overline K\colon\quad
	\radialcone_{\overline K}(y) = \radialcone_K(y)\cap\bar\lambda\anni.
\end{equation}
The upcoming lemma provides some calculus rules for cones
and will become handy at several places in this paper.
\begin{lemma}
	\label{lem:fun_with_cones}
	\hfill
	\begin{enumerate}
	\item\label{item:fun_with_cones_general}
	Let $y \in K$ and $\bar\lambda \in \normalcone_K(y)$ be given.
	Then we have the identities
	\begin{equation*}
		\parens*{ \tangentcone_K(y) \cap \bar\lambda\anni }\polar
		=
		\wsclosure\parens*{ \normalcone_K(y) + \R_- \bar\lambda }
		=
		\wsclosure\parens*{ \radialcone_{\normalcone_K(y)}(\bar\lambda)}
		.
	\end{equation*}
	If, additionally, $Y$ is reflexive,
	we get
	\begin{equation*}
		\parens*{ \tangentcone_K(y) \cap \bar\lambda\anni }\polar
		=
		\tangentcone_{\normalcone_K(y)}(\bar\lambda).
	\end{equation*}
	\item\label{item:fun_with_cones_polyhedric}
	Let $(\bar x,\bar\lambda)\in X\times Y\dualspace$ be a KKT pair of \eqref{eq:problem},
	and fix $y\in \overline K$,
	where $\overline K$ is the set defined in \eqref{eq:bar_K}.
	Then $\bar\lambda\in\normalcone_K(y)$, 
	and if $K$ is polyhedric w.r.t.\ $(y,\bar\lambda)$, we have the identity
	\[
		\normalcone_{\overline K}(y)
		=
		\wsclosure\parens*{ \normalcone_K(y) + \R_- \bar\lambda }
		=
		\wsclosure\parens*{ \radialcone_{\normalcone_K(y)}(\bar\lambda)}
		.
	\]
	If, additionally, $Y$ is reflexive,
	we get
	\begin{equation*}
		\normalcone_{\overline K}(y)
		=
		\tangentcone_{\normalcone_K(y)}(\bar\lambda)
		.
	\end{equation*}
	\end{enumerate}
\end{lemma}
\begin{proof}
	To start, let us prove the first statement of assertion \ref{item:fun_with_cones_general}.
	We equip $Y$ with its strong topology and $Y\dualspace$ with its weak-$\star$ topology.
	Consequently,
	we can apply \cite[(2.32)]{BonnansShapiro2000}
	and obtain
	\begin{align*}
		\parens*{ \tangentcone_K(y) \cap \bar\lambda\anni }\polar
		=
		\wsclosure\parens*{ \parens*{\tangentcone_K(y)}\polar + (\bar\lambda\anni)\polar }
		=
		\wsclosure\parens*{ \normalcone_K(y) + \R \bar\lambda }
		.
	\end{align*}
	Due to $\bar\lambda \in \normalcone_K(y)$,
	we get
	$\normalcone_K(y) + \R \bar\lambda = \normalcone_K(y) + \R_- \bar\lambda$
	and, thus, the first identity has been shown.
	The second identity follows as $\normalcone_K(y)$ is a cone.
	Hence, the first statement of assertion \ref{item:fun_with_cones_general}
	has already been verified.
	If $Y$ is reflexive, we can equip $Y\dualspace$
	with its strong topology.
	Consequently, the weak-$\star$ closure above can be replaced by the strong closure.
	Hence, we obtain
	\begin{equation*}
		\parens*{ \tangentcone_K(y) \cap \bar\lambda\anni }\polar
		=
		\cl\parens*{ \radialcone_{\normalcone_K(y)}(\bar\lambda) }
		=
		\tangentcone_{\normalcone_K(y)}(\bar\lambda)
		,
	\end{equation*}
	verifying the second statement of assertion \ref{item:fun_with_cones_general}.
	
	Let us now proceed with the proof of assertion \ref{item:fun_with_cones_polyhedric}.
	To start, we note that $\bar\lambda\in\normalcone_K(g(\bar x))$ and $y\in\overline K$
	yield $\bar\lambda\in\normalcone_K(y)$.
	As $K$ is polyhedric w.r.t.\ $(y,\bar\lambda)$,
	we have $\cl(\radialcone_K(y)\cap\bar\lambda\anni)=\tangentcone_K(y)\cap\bar\lambda\anni$.
	Thus, with the aid of \eqref{eq:radialcone_to_bar_K}, we find
	\begin{equation*}
		\normalcone_{\overline K}(y)
		=
		\parens*{\tangentcone_{\overline K}(y)}\polar
		=
		\parens*{\radialcone_{\overline K}(y)}\polar
		=
		\parens*{\radialcone_K(y)\cap\bar\lambda\anni}\polar
		=
		\parens*{\tangentcone_K(y)\cap\bar\lambda\anni}\polar,
	\end{equation*}
	and the claimed identities follow from assertion \ref{item:fun_with_cones_general}.
\end{proof}

Again, let us fix a KKT pair $(\bar x,\bar\lambda)\in X\times Y\dualspace$ of \eqref{eq:problem}.
We note that \eqref{eq:sRZKC} is equivalent to each of the following conditions:
\begin{subequations}\label{eq:sRZKC_alt}
	\begin{align}
		\label{eq:sRZKC_alt_i}
			\exists\rho>0\colon\quad
			B^Y_\rho(0)&\subset g'(\bar x) B_1^X(0) - (\overline K-\set{g(\bar x)})\cap B_1^Y(0),
			\\
		\label{eq:sRZKC_alt_ii}
			0&\in\interior(\set{g(\bar x)}+g'(\bar x)X - \overline{K}),
			\\
		\label{eq:sRZKC_alt_iii}
			Y &= g'(\bar x)X - \radialcone_{\overline K}(g(\bar x)).
	\end{align}
\end{subequations}
Indeed, thanks to \eqref{eq:radialcone_to_bar_K} and $g(\bar x)\in\overline K$,
\eqref{eq:sRZKC} and \eqref{eq:sRZKC_alt_iii} are equivalent.
Furthermore, the implications
[\eqref{eq:sRZKC_alt_i} $\Longrightarrow$ \eqref{eq:sRZKC_alt_ii} $\Longrightarrow$ \eqref{eq:sRZKC_alt_iii}]
are obvious, and [\eqref{eq:sRZKC_alt_iii} $\Longrightarrow$ \eqref{eq:sRZKC_alt_i}]
follows from \cite[Theorem~2.1]{ZoweKurcyusz1979}. 
Clearly, \eqref{eq:sRZKC} is sufficient for \eqref{eq:RZKCQ}.

We are also concerned with the conditions
\begin{subequations}\label{eq:weak_sRZKC}
	\begin{align}
		\label{eq:weak_sRZKC_i}
		Y &= g'(\bar x)X - \tangentcone_K(g(\bar x))\cap\bar\lambda\anni,
		\\
		\label{eq:weak_sRZKC_ii}
		Y &= \cl\parens*{g'(\bar x)X - \tangentcone_K(g(\bar x))\cap\bar\lambda\anni}.
	\end{align}
\end{subequations}
Obviously, we have the implications
[\eqref{eq:sRZKC} $\Longrightarrow$ \eqref{eq:weak_sRZKC_i} $\Longrightarrow$ \eqref{eq:weak_sRZKC_ii}],
and all three conditions are equivalent if $Y$ is finite dimensional
and $K$ is polyhedric.
Indeed, in the latter case, the definition of polyhedricity yields
\[
	\cl\parens*{\radialcone_K(g(\bar x))\cap\bar\lambda\anni}
	=
	\tangentcone_K(g(\bar x))\cap\bar\lambda\anni
\]
which guarantees that \eqref{eq:weak_sRZKC_ii} equals
\[ 	
	Y = \cl\parens*{g'(\bar x)X - \radialcone_K(g(\bar x))\cap\bar\lambda\anni},	
\]
and the latter is equivalent to \eqref{eq:sRZKC}
whenever $Y$ is finite dimensional,
see, e.g., \cite[Theorem~2.17]{BonnansShapiro2000}.
In general, however, these relations might be strict.
Indeed, \cref{ex:weak_sRZKC_vs_RZKCQ} from below
presents a situation where \eqref{eq:weak_sRZKC_i} is valid while \eqref{eq:sRZKC} fails.
Furthermore, if $g'(\bar x)$ has a dense range and $K=\set{0}$,
then \eqref{eq:weak_sRZKC_ii} holds while \eqref{eq:weak_sRZKC_i} is violated,
and the particular choice of $\bar\lambda$ is, obviously, not relevant in this situation.

\begin{example}\label{ex:weak_sRZKC_vs_RZKCQ}
	Let us investigate \eqref{eq:problem}, where
	\begin{align*}
		X &= L^2(0,1),
		&
		Y &= L^2(0,1) \times L^2(0,1),
		&
		K &= L^2(0,1)_+ \times L^2(0,1)_+,
		\\
		\bar y & = g(\bar x) = (\zerofunction, \onefunction),
		&
		f'(\bar x) &= \onefunction,
		&
		g'(\bar x) x &= (x,\zerofunction)
		.
	\end{align*}
	We note that the precise choice of $\bar x\in X$ is not relevant here.
	Elementary calculations yield
	\begin{equation}\label{eq:some_cones_in_L2}
	\begin{aligned}
		\radialcone_K(\bar y)
		&=
		L^2(0,1)_+ \times \parens*{ L^2(0,1)_+ + L^\infty(0,1) }
		,
		\\
		\tangentcone_K(\bar y)
		&=
		L^2(0,1)_+ \times L^2(0,1)
		,
		\\
		\normalcone_K(\bar y)
		&=
		L^2(0,1)_- \times \set{\zerofunction}
		.
	\end{aligned}
	\end{equation}
	On the one hand, 
	we easily see that \eqref{eq:RZKCQ} is violated in the situation at hand.
	On the other hand, we find $\bar\lambda=(-\onefunction,\zerofunction)\in\Lambda(\bar x)$.
	We observe that
	\begin{equation}\label{eq:some_annihilated_cones_in_L2}
	\begin{aligned}
		\radialcone_K(\bar y) \cap \bar\lambda\anni
		&=
		\set{\zerofunction} \times \parens*{ L^2(0,1)_+ + L^\infty(0,1) }
		,
		\\
		\tangentcone_K(\bar y) \cap \bar\lambda\anni
		&=
		\set{\zerofunction} \times L^2(0,1)
		,
	\end{aligned}
	\end{equation}
	showing, particularly, that \eqref{eq:weak_sRZKC_i} holds
	while \eqref{eq:sRZKC} fails.
\end{example}

Whenever $K$ is a cone,
\cite[Proposition~2.1]{Shapiro1997} provides a characterization 
of the uniqueness of multipliers.
Furthermore, it follows from \cite[Theorem~2.2]{Shapiro1997} that \eqref{eq:weak_sRZKC_ii}
already implies $\Lambda(\bar x)=\set{\bar\lambda}$,
and that the converse is true if $\radialcone_{\normalcone_K(g(\bar x))}(\bar\lambda)$
is weakly-$\star$ closed.
As it turns out, all these facts remain true if $K$ is merely a closed convex set.
In the following result, we summarize these observations,
and present a proof in order to keep the presentation self-contained.

\begin{theorem}\label{thm:uniqueness}
	Let $(\bar x,\bar\lambda)\in X\times Y\dualspace$
	be a KKT pair of \eqref{eq:problem}.
	Then the following assertions hold.
	\begin{enumerate}
		\item\label{item:uniqueness_char}
			Condition $\Lambda(\bar x)=\set{\bar\lambda}$ is equivalent to
			\begin{equation}\label{eq:uniqueness_char}
				\ker g'(\bar x)\adjoint 
				\cap 
				\radialcone_{\normalcone_K(g(\bar x))}(\bar\lambda)
				=
				\set{ 0 }
				.
			\end{equation}
		\item\label{item:uniqueness_weak_sRZKC_suff}
			Condition \eqref{eq:weak_sRZKC_ii} is sufficient for \eqref{eq:uniqueness_char}.
		\item\label{item:uniqueness_weak_sRZKC_nec}
			Let $\radialcone_{\normalcone_K(g(\bar x))}(\bar\lambda)$ be weakly-$\star$ closed.
			Then condition \eqref{eq:weak_sRZKC_ii} is equivalent to \eqref{eq:uniqueness_char}.
	\end{enumerate}
\end{theorem}
\begin{proof}
	For the proof of assertion~\ref{item:uniqueness_char},
	we verify both implications separately.\\
	$[\Longrightarrow]$ Pick $\mu\in\ker g'(\bar x)\adjoint\cap \radialcone_{\normalcone_K(g(\bar x))}(\bar\lambda)$ arbitrarily. Then, by definition of the radial cone and the fact that
	$\normalcone_K(g(\bar x))$ is a cone, we find $\lambda\in\normalcone_K(g(\bar x))$
	and $\alpha\geq 0$ such that $\mu=\lambda-\alpha\bar\lambda$.
	If $\alpha>0$ holds, 
	we have $\mu/\alpha+\bar\lambda=\lambda/\alpha\in\normalcone_K(g(\bar x))$
	since $\normalcone_K(g(\bar x))$ is a cone.
	Furthermore,
	\[
		\LL'_x(\bar x,\mu/\alpha + \bar\lambda)
		=
		\LL'_x(\bar x,\bar\lambda) + \frac1\alpha g'(\bar x)\adjoint\mu
		=
		0
	\]
	holds true, so that $\mu/\alpha + \bar\lambda\in\Lambda(\bar x)$ is obtained.
	Hence, due to $\Lambda(\bar x)=\set{\bar\lambda}$ and $\alpha>0$,
	$\mu=0$ follows.
	If $\alpha=0$, we have $\mu=\lambda\in\normalcone_K(g(\bar x))$,
	and $\mu+\bar\lambda\in\normalcone_K(g(\bar x))$ follows
	as $\normalcone_K(g(\bar x))$ is a convex cone.
	As we have
	\[
		\LL'_x(\bar x,\mu+\bar\lambda)
		=
		\LL'_x(\bar x,\bar\lambda) + g'(\bar x)\adjoint\mu 
		=
		0,
	\]
	$\mu+\bar\lambda\in\Lambda(\bar x)$ follows,
	and due to $\Lambda(\bar x)=\set{\bar\lambda}$,
	we also end up with $\mu=0$.
	Taking both cases together, we have shown validity of \eqref{eq:uniqueness_char}.
	\\
	$[\Longleftarrow]$ Pick some arbitrary multiplier $\lambda\in\Lambda(\bar x)$
	and consider $\mu:=\lambda-\bar\lambda$.
	Then, due to $\bar\lambda\in\Lambda(\bar x)$, we obviously have $\mu\in\ker g'(\bar x)\adjoint$,
	and
	\[
		\mu \in \normalcone_K(g(\bar x))-\set{\bar\lambda}
		\subset
		\normalcone_K(g(\bar x)) + \R_-\bar\lambda
		=
		\radialcone_{\normalcone_K(g(\bar x))}(\bar\lambda)
	\]
	is obtained since $\normalcone_K(g(\bar x))$ is a cone containing $\bar\lambda$.
	According to \eqref{eq:uniqueness_char},
	$\mu=0$ follows, i.e., $\lambda=\bar\lambda$ has been shown.
	
	For the proof of assertions \ref{item:uniqueness_weak_sRZKC_suff} and \ref{item:uniqueness_weak_sRZKC_nec},
	we note that \cref{lem:fun_with_cones}\,\ref{item:fun_with_cones_general} yields
	\begin{equation}\label{eq:some_polar_relation}
		\begin{aligned}
		\parens*{%
			g'(\bar x)X - \tangentcone_K(g(\bar x))\cap\bar\lambda\anni
		}\polar
		&=
		\parens*{g'(\bar x)X}\polar
		\cap
		\parens{-\tangentcone_K(g(\bar x))\cap\bar\lambda\anni}\polar
		\\
		&=
		\ker g'(\bar x)\adjoint
		\cap
		\parens*{-\wsclosure\parens*{\radialcone_{\normalcone_K(g(\bar x))}(\bar\lambda)}}.
		\end{aligned}
	\end{equation}
	On the one hand, if \eqref{eq:weak_sRZKC_ii} holds, 
	polarization gives
	\begin{equation}\label{eq:polar_weak_sRZKC}
		\set{ 0 } 
		= 
		\parens*{%
			g'(\bar x)X - \tangentcone_K(g(\bar x))\cap\bar\lambda\anni
		}\polar,
	\end{equation}
	and \eqref{eq:some_polar_relation} implies validity of \eqref{eq:uniqueness_char},
	i.e., assertion \ref{item:uniqueness_weak_sRZKC_suff} holds.
	On the other hand, if \eqref{eq:uniqueness_char} holds while
	$\radialcone_{\normalcone_K(g(\bar x))}(\bar\lambda)$ is 
	is weakly-$\star$ closed, then \eqref{eq:uniqueness_char}
	and \eqref{eq:some_polar_relation} show that
	\eqref{eq:polar_weak_sRZKC} is valid.
	As $g'(\bar x)X-\tangentcone_K(g(\bar x))\cap\bar\lambda\anni$
	is a convex cone, validity of \eqref{eq:weak_sRZKC_ii} follows, e.g.,
	from \cite[Proposition~2.40]{BonnansShapiro2000}.
	This proves that assertion \ref{item:uniqueness_weak_sRZKC_nec} is true.
\end{proof}

It should be observed that 
conditions \eqref{eq:RZKCQ} and \eqref{eq:weak_sRZKC_i} are, in general, not related.
Indeed, as \eqref{eq:RZKCQ} does not guarantee the uniqueness of multipliers,
it cannot be sufficient for \eqref{eq:weak_sRZKC_i} in general.
Conversely, \cref{ex:weak_sRZKC_vs_RZKCQ} depicts that validity of \eqref{eq:weak_sRZKC_i}
does not necessarily imply that \eqref{eq:RZKCQ} holds.
Let us mention that
the validity of \eqref{eq:weak_sRZKC_ii} already implies
\[
	Y = \cl\parens*{g'(\bar x)X - \tangentcone_K(g(\bar x))},
\] 
and the latter is equivalent to \eqref{eq:RZKCQ}
whenever $Y$ is finite dimensional,
see, e.g., \cite[Proposition~2.97]{BonnansShapiro2000}.
Hence, the situation
in \cref{ex:weak_sRZKC_vs_RZKCQ} has to be fully attributed to 
the infinite-dimensional setting therein.

The following example illustrates that none of the conditions
\eqref{eq:sRZKC}, \eqref{eq:weak_sRZKC_i}, or \eqref{eq:weak_sRZKC_ii}
is necessary for the uniqueness of the multiplier.
We note that an example illustrating this has not been presented in \cite{Shapiro1997}.

\begin{example}\label{ex:multiplier_unique_but_nothing_else}
	Let us investigate \eqref{eq:problem} where
	\[
		\begin{aligned}
			X&=\R^2,&\quad	Y&=\R^3,& &&
			\\
			f(x)&=\sqrt 2 x_1 + x_2,&\quad g(x)&=(x_1,x_1,x_2),&\quad K&=\set{y\in\R^3 \given (y_1^2+y_2^2)^{1/2}\leq y_3}.&
		\end{aligned}
	\]
	Then \eqref{eq:problem} is a linear second-order cone problem.
	It is not hard to check that the feasible set of \eqref{eq:problem} equals $\set{x\in\R^2 \given \sqrt 2|x_1|\leq x_2}$,
	so $\bar x=(0,0)$ is a minimizer of \eqref{eq:problem}.
	As we have $\bar y = g(\bar x) = (0,0,0)$, we find
	\[
		\radialcone_K(\bar y) = \tangentcone_K(\bar y) = K,\qquad
		\normalcone_K(\bar y) = K\polar = \set{\lambda\in\R^3 \given (\lambda_1^2+\lambda_2^2)^{1/2}\leq-\lambda_3}.
	\]
	Hence, a simple calculation reveals that \eqref{eq:RZKCQ} is valid at $\bar x$.
	Let us investigate the associated KKT conditions
	\[
		(0,0) = (\sqrt 2,1) + (\lambda_1 + \lambda_2,\lambda_3),\quad \lambda\in K\polar.
	\]
	From the second component we infer $\lambda_3=-1$,
	so, for the first component, we need $\lambda_1+\lambda_2=-\sqrt 2$ and $\lambda_1^2+\lambda_2^2\leq 1$
	which is only possible if $\lambda_1=\lambda_2=-\sqrt 2/2$.
	Hence, $\bar\lambda=(-\sqrt 2/2,-\sqrt 2/2,-1)$ is the uniquely determined multiplier associated with $\bar x$.
	
	Noting that
	\[
		\radialcone_K(\bar y)\cap\bar\lambda\anni
		=
		\tangentcone_K(\bar y)\cap\bar\lambda\anni
		=
		K\cap\bar\lambda\anni
		=
		\set{(t,t,-\sqrt 2 t)\in\R^3 \given t\leq 0},
	\]
	it is obvious that \eqref{eq:weak_sRZKC_ii} is violated,
	and so are \eqref{eq:weak_sRZKC_i} and \eqref{eq:sRZKC},
	i.e., none of these conditions is necessary for the uniqueness of the multiplier.
\end{example}

\section{Stability of Lagrange multipliers}\label{sec:main}

Throughout the section, 
we investigate a given KKT pair $(\bar x,\bar\lambda)\in X\times Y\dualspace$
which satisfies the isolated calmness type condition \eqref{eq:isol_calmness_ups} 
for the restricted multiplier mapping $\Upsilon_{\bar x}$.

Let us recall that validity of \eqref{eq:sRZKC}
guarantees that \eqref{eq:isol_calmness_ups} holds
according to \cite[Proposition~4.47]{BonnansShapiro2000}.
Here, we are concerned with the question whether
the converse might also be true---potentially under additional assumptions.

To start, we show that \eqref{eq:isol_calmness_ups} 
is always sufficient for \eqref{eq:RZKCQ}.

\begin{theorem}\label{thm:IC_gives_RZKCQ}
	Let $(\bar x,\bar\lambda)\in X\times Y\dualspace$ be a KKT pair of \eqref{eq:problem}
	and assume that \eqref{eq:isol_calmness_ups} is valid.
	Then \eqref{eq:RZKCQ} holds.
\end{theorem}
\begin{proof}
	To start, let us note that, 
	according to \cite[Proof of Theorem~2.1]{ZoweKurcyusz1979},
	\eqref{eq:RZKCQ} is equivalent to the existence of $\rho>0$
	such that 
	\begin{equation*}
			B^Y_\rho(0) \subset \cl\parens*{ g'(\bar x) B^X_1(0) - (K-\set{g(\bar x)}) \cap B^Y_1(0) }.
		\end{equation*}
	In order to prove the claim, let us argue by contradiction.
	If \eqref{eq:RZKCQ} is violated, 
	then,
	for each $\rho\in(0,1)$, 
	we find $y_\rho\in B^Y_\rho(0)$ such that
	\[
		y_\rho \notin \cl\parens*{ g'(\bar x) B^X_1(0) - (K-\set{g(\bar x)}) \cap B^Y_1(0) }.
	\]
	Noting that the set on the right is convex and closed, 
	$y_\rho$ can be separated from it,
	see, e.g., \cite[Theorem~1.5.9]{Schirotzek2007},
	i.e., we find $y_\rho^\star\in Y\dualspace$ with $\norm{y_\rho^\star}_{Y\dualspace}=\rho^{1/2}$ such that
	\begin{equation}\label{eq:separability_RZKCQ}
		\forall x\in B_1^X(0),\,\forall y\in K\cap B_1^Y(g(\bar x))\colon\quad
		\dual{y_\rho^\star}{y_\rho}_Y
		>
		\dual{g'(\bar x)\adjoint y_\rho^\star}{x}_X - \dual{y_\rho^\star}{y-g(\bar x)}_Y.
	\end{equation}
	Note that this yields $0 < \dual{y_\rho^\star}{y_\rho}_Y \leq \rho^{3/2}$.
	Testing \eqref{eq:separability_RZKCQ} by $y\coloneqq g(\bar x)$, we find
	\begin{equation}\label{eq:some_estimate}
		\norm{g'(\bar x)\adjoint y_\rho^\star}_{X\dualspace}\leq\rho^{3/2}
	\end{equation}
	according to definition of the norm in $X\dualspace$.
	Using $x\coloneqq 0$ in \eqref{eq:separability_RZKCQ}, we find that
	\begin{equation*}
		\forall y\in K\cap B_1^Y(g(\bar x))\colon\quad
		\rho^{3/2}
		\ge
		\dual{y_\rho^\star}{y_\rho}_Y
		>
		- \dual{y_\rho^\star}{y-g(\bar x)}_Y.
	\end{equation*}
	Next, we add the valid inequality
	\begin{equation*}
		\forall y\in K \colon\quad
		0
		\ge
		\dual{\bar\lambda}{y - g(\bar x)}_Y,
	\end{equation*}
	following from $\bar\lambda\in\normalcone_K(g(\bar x))$, to obtain
	\begin{equation*}
		\forall y\in K\cap B_1^Y(g(\bar x))\colon\quad
		\rho^{3/2}
		\ge
		\dual{\bar\lambda - y_\rho^\star}{y-g(\bar x)}_Y.
	\end{equation*}
	Consequently,
	$g(\bar x)$ is a $\rho^{3/2}$-minimal point of the function $\varphi\colon Y\to\R\cup\set{\infty}$
	given by
	\[
		\forall y\in Y\colon\quad
		\varphi(y) \coloneqq \iota_{K\cap B^Y_1(g(\bar x))}(y) + \dual{y_\rho^\star - \bar\lambda}{y-g(\bar x)}_Y.
	\]
	Thus, we may apply Ekeland's variational principle, 
	see, e.g., \cite[Theorem~1.1]{Ekeland1974},
	in order to find $\bar y_\rho\in K\cap B_1^Y(g(\bar x))$ satisfying
	\begin{equation}\label{eq:bound_on_norm}
		\norm{\bar y_\rho - g(\bar x)}_Y
		\leq
		\rho^{3/4}
		<
		1
	\end{equation}
	and such that $\bar y_\rho$ is the uniquely determined minimizer 
	of $\varphi + \rho^{3/4} \norm{\cdot - \bar y_\rho}_Y$.
	Applying some calculus rules for the subdifferential,
	see, e.g., \cite[Remark~4.4.2, Propositions~4.5.1 and~4.6.2]{Schirotzek2007}, we find
	\[
		0 \in \normalcone_K(\bar y_\rho) + \set{ y_\rho^\star - \bar\lambda } + \rho^{3/4} B^{Y\dualspace}_1(0),
	\]
	where we also used \eqref{eq:bound_on_norm}.
	Thus, there exists $\lambda_\rho\in\normalcone_K(\bar y_\rho)$
	such that 
	\begin{equation}
		\label{eq:norm_of_lambda_and_y_rho_dualspace}
		\norm{\lambda_\rho + y_\rho^\star -\bar\lambda }_{Y\dualspace}\leq \rho^{3/4}
		.
	\end{equation}
	Consequently, due to $\norm{y\dualspace_\rho}_{Y\dualspace}=\rho^{1/2}$, the estimates
	\begin{equation}
		\label{eq:norm_of_some_lambda_difference}
		\rho^{1/2} - \rho^{3/4}
		\le
		\norm{ \lambda_\rho - \bar\lambda }_{Y\dualspace}
		\le
		\rho^{1/2} + \rho^{3/4}
	\end{equation}
	are valid.
	
	We have $\lambda_\rho\in\normalcone_K(g(\bar x) + (\bar y_\rho - g(\bar x)))$.
	Furthermore, we find
	\[
		\LL'_x(\bar x,\lambda_\rho)
		=
		\LL'_x(\bar x,\bar\lambda) + g'(\bar x)\adjoint(\lambda_\rho -\bar\lambda)
		=
		g'(\bar x)\adjoint(\lambda_\rho-\bar\lambda).
	\] 
	Hence, we obtain
	\begin{equation*}
		\lambda_\rho
		\in
		\Upsilon_{\bar x}(
			g'(\bar x)\adjoint(\lambda_\rho - \bar\lambda),
			\bar y_\rho - g(\bar x)
		).
	\end{equation*}
	Thanks to \eqref{eq:bound_on_norm} and \eqref{eq:norm_of_some_lambda_difference},
	we can invoke \eqref{eq:isol_calmness_ups} for $\rho$ small enough
	and obtain
	\begin{align*}
		\rho^{1/2} - \rho^{3/4}
		&
		\le
		\norm{\lambda_\rho - \bar\lambda}_{Y\dualspace}
		\\&
		\le
		c \parens*{
			\norm{
				g'(\bar x)\adjoint(\lambda_\rho-\bar\lambda)
			}_{X\dualspace}
			+
			\norm{\bar y_\rho - g(\bar x)}_Y
		}
		\\&
		\le
		c \parens*{
			\norm{
				g'(\bar x)\adjoint(\lambda_\rho+ y_\rho\dualspace -\bar\lambda)
			}_{X\dualspace}
			+
			\norm{
				g'(\bar x)\adjoint y_\rho\dualspace
			}_{X\dualspace}
			+
			\rho^{3/4}
		}
		\\&
		\le
		c
		\parens*{
			\norm{g'(\bar x)}_{L(X,Y)} \rho^{3/4}
			+
			\rho^{3/2}
			+
			\rho^{3/4}
		}
		,
	\end{align*}
	where we used \eqref{eq:some_estimate}, \eqref{eq:norm_of_lambda_and_y_rho_dualspace},
	and \eqref{eq:norm_of_some_lambda_difference}.
	Division by $\rho^{1/2}$ and taking the limit $\rho\downarrow 0$ yield a contradiction.
\end{proof}

Our next result shows that \eqref{eq:isol_calmness_ups}
is always sufficient for \eqref{eq:weak_sRZKC_i}
under some additional assumption.
Recall from \cref{sec:QCs} that this observation
is independent from \cref{thm:IC_gives_RZKCQ}.

\begin{theorem}
	\label{thm:yet_another_theorem}
	Let $(\bar x,\bar\lambda)\in X\times Y\dualspace$ be a KKT pair of \eqref{eq:problem}
	and assume that \eqref{eq:isol_calmness_ups} is valid.
	Furthermore, assume that one of the following conditions holds.
	\begin{enumerate}
		\item\label{item:Yref}
			The space $Y$ is reflexive.
		\item\label{item:radial_cone_weak_star_closed}
			The cone $\radialcone_{\normalcone_K(g(\bar x))}(\bar\lambda)$ 
			is weakly-$\star$ closed.
		\item\label{item:range_finite_dimensional}
			The range of $g'(\bar x)$ is finite dimensional. 
	\end{enumerate}
	Then \eqref{eq:weak_sRZKC_i} holds.
\end{theorem}
\begin{proof}
	We use $\upsilon = 0$ in \eqref{eq:isol_calmness_ups}.
	This yields
	\begin{equation}
		\label{eq:isol_calmness_ups0}
		\forall \xi \in B_\varepsilon^{X\dualspace}(0),\,
		\forall \lambda\in\Upsilon_{\bar x}(\xi,0)
		\cap 
		B_\delta^{Y\dualspace}(\bar\lambda)
		\colon
		\qquad
		\norm{\lambda-\bar \lambda}_{Y\dualspace}
		\leq
		c \norm{\xi}_{X\dualspace}.
	\end{equation}
	Note that for any $\lambda \in Y\dualspace$, we have
	$\LL'_x(\bar x,\lambda) = g'(\bar x)\adjoint (\lambda - \bar\lambda)$
	from $\bar\lambda\in\Lambda(\bar x)$.
	Consequently,
	for all
	$\lambda \in \normalcone_K(g(\bar x))$
	with
	$\norm{\lambda - \bar\lambda}_{Y\dualspace} \le \delta$
	and
	$\norm{ g'(\bar x)\adjoint (\lambda - \bar\lambda) }_{X\dualspace} \le \varepsilon$,
	we can use
	$\xi = \LL'_x(\bar x, \lambda)$
	in \eqref{eq:isol_calmness_ups0}
	and this yields
	\begin{equation*}
		\norm{\lambda - \bar\lambda}_{Y\dualspace}
		\le
		c
		\norm{ g'(\bar x)\adjoint (\lambda - \bar\lambda) }_{X\dualspace}
		.
	\end{equation*}
	A simple scaling argument implies
	\begin{equation}\label{eq:estimate_on_radial_cone}
		\forall \mu \in \radialcone_{\normalcone_K(g(\bar x))}(\bar\lambda)
		\colon
		\qquad
		\norm{ \mu }_{Y\dualspace}
		\le
		c
		\norm{ g'(\bar x)\adjoint \mu }_{X\dualspace}
		.
	\end{equation}
	
	In the presence of \ref{item:Yref},
	we first note that, by density, \eqref{eq:estimate_on_radial_cone} yields
	\begin{equation*}
		\forall \mu \in \tangentcone_{\normalcone_K(g(\bar x))}(\bar\lambda)
		\colon
		\qquad
		\norm{ \mu }_{Y\dualspace}
		\le
		c
		\norm{ g'(\bar x)\adjoint \mu }_{X\dualspace}
		.
	\end{equation*}
	From \cref{lem:fun_with_cones}\,\ref{item:fun_with_cones_general}, we get
	\begin{equation}\label{eq:estimate_on_polar_critical_cone}
		\forall \mu \in
		\parens*{ \tangentcone_K(\bar y) \cap \bar\lambda\anni }\polar
		\colon
		\qquad
		\norm{ \mu }_{Y\dualspace}
		\le
		c
		\norm{ g'(\bar x)\adjoint \mu }_{X\dualspace}
		.
	\end{equation}
	Using $A := g'(\bar x)$ and $C := \tangentcone_K(\bar y) \cap \bar\lambda\anni$,
	the claim follows from \cref{thm:gen_open_2}.
	
	Whenever \ref{item:radial_cone_weak_star_closed} holds,
	\eqref{eq:estimate_on_radial_cone} is the same as
	\begin{equation}\label{eq:estimate_on_weak_star_closure_of_radial_cone}
		\forall \mu \in \wsclosure\parens*{\radialcone_{\normalcone_K(g(\bar x))}(\bar\lambda)}
		\colon
		\qquad
		\norm{ \mu }_{Y\dualspace}
		\le
		c
		\norm{ g'(\bar x)\adjoint \mu }_{X\dualspace},
	\end{equation}
	and
	\cref{lem:fun_with_cones}\,\ref{item:fun_with_cones_general} directly yields that
	\eqref{eq:estimate_on_polar_critical_cone} is true,
	so that the desired result is, again, valid due to \cref{thm:gen_open_2}.	
	
	Finally, we are going to verify that, in the presence of \ref{item:range_finite_dimensional},
	\eqref{eq:estimate_on_weak_star_closure_of_radial_cone} is true as well,
	which yields the claim as mentioned above.
	Therefore, we note that the additional assumption guarantees that
	$g'(\bar x)\adjoint$ has the representation
	\begin{equation*}
		\forall y\dualspace\in Y\dualspace\colon\quad
		g'(\bar x)\adjoint y\dualspace 
		= 
		\sum_{k = 1}^n \dual{y\dualspace}{y_k}_Y\, x_k\dualspace
	\end{equation*}
	for some $n \in \N$ as well as $y_1,\ldots,y_n \in Y$ and $x_1\dualspace,\ldots,x_n\dualspace \in X\dualspace$,
	see, e.g., \cite[Section~28]{Heuser1986}.
	Consequently,
	$g'(\bar x)\adjoint$ sends weakly-$\star$ convergent nets to strongly convergent
	ones.
	Hence, recalling that each
	$\mu\in\wsclosure\bigl(\radialcone_{\normalcone_K(g(\bar x))}(\bar\lambda)\bigr)$
	can be represented as the limit of a weakly-$\star$ convergent net
	from $\radialcone_{\normalcone_K(g(\bar x))}(\bar\lambda)$
	while $\norm{ \cdot}_{Y\dualspace}$ is weakly-$\star$ lower semicontinuous,
	\eqref{eq:estimate_on_weak_star_closure_of_radial_cone} follows
	from \eqref{eq:estimate_on_radial_cone}.
\end{proof}

Let us note that condition \ref{item:range_finite_dimensional} from \cref{thm:yet_another_theorem}
is trivially satisfied if $X$ is finite dimensional.
Furthermore, observe the above proof only utilizes \eqref{eq:isol_calmness_ups}
with $\upsilon = 0$.
Finally, we would like to mention that \cref{ex:multiplier_unique_but_nothing_else}
and \cref{thm:yet_another_theorem} illustrate that \eqref{eq:isol_calmness_ups}
is not necessary for the uniqueness of the multiplier.

The next example shows
that \eqref{eq:weak_sRZKC_i}
is not sufficient for \eqref{eq:isol_calmness_ups}.
It is inspired by \cite[Example~4.22]{Wachsmuth2019}.
\begin{example}
	\label{ex:this_condition_is_not_enough}
	Let us set
	\begin{equation*}
		X = \R^2
		,\quad
		Y = C([-1,1])
		,\quad
		K =
		\set{
			\varphi \in Y
			\given
			\forall s \in [-1,1] \colon
			0 \le \varphi(s) \le 1 + s^2
		}
		.
	\end{equation*}
	We further define the operator $A \in L(X, Y)$ via
	\begin{equation*}
		\forall x \in X ,\, \forall s \in [-1,1] \colon
		\qquad
		(A x)(s) = x_1 + x_2 s
	\end{equation*}
	and we set $g(x) = A x$, which implies $g'(x) = A$ for all $x \in X$.
	A precise definition of $f$ is not required in this example.

	We consider the feasible point $\bar x = (1, 0) \in \R^2$
	and set $\bar y = g(\bar x) = \onefunction$.
	Note that
	\begin{equation*}
		\tangentcone_K(\bar y)
		=
		\set{
			\varphi \in Y
			\given
			\varphi(0) \le 0
		}
		.
	\end{equation*}
	Further,
	$\bar\lambda = \delta_0$
	is the Dirac measure at the point $0$.

	Clearly,
	$\radialcone_K(\bar y)$ contain all nonpositve functions from $Y$.
	Since $g'(\bar x)X$ contain the constant functions, this implies
	that \eqref{eq:RZKCQ} is satisfied.
	Further, it is easy to check that \eqref{eq:weak_sRZKC_i} holds
	as we have
	\begin{equation*}
		\tangentcone_K(\bar y) \cap \bar\lambda\anni
		=
		\set{
			\varphi \in Y
			\given
			\varphi(0) = 0
		}
	\end{equation*}
	and since $g'(\bar x) X$ contains the constant functions.
	However, \eqref{eq:isol_calmness_ups} is violated,
	as it will be shown next.

	Let us check that \eqref{eq:isol_calmness_ups} is not satisfied.
	To this end, let $\varepsilon \in (0,\tfrac12)$
	be given and choose $\upsilon \in Y$ such that
	$\bar y + \upsilon \in K$ and
	\begin{equation*}
		\forall s \in [-\varepsilon,\varepsilon]\colon
		\qquad
		(\bar y + \upsilon)(s) = 1 + s^2
		.
	\end{equation*}
	Note that this is possible with $\norm{\upsilon}_Y \le \varepsilon^2$.
	Next, we denote by $\lambda$ the signed regular measure with
	\begin{equation*}
		\forall B \in \borel([-1,1]) \colon
		\quad
		\lambda(B)
		=
		\int_{B \cap [-\varepsilon, \varepsilon]} 1 \d s
		+
		(1 - 2\varepsilon) \delta_0(B).
	\end{equation*}
	By definition of the normal cone, it is clear that $\lambda \in \normalcone_K(\bar y + \upsilon)$.
	Consequently,
	$\lambda \in \Upsilon_{\bar x}( g'(\bar x)\adjoint (\lambda - \bar\lambda) , \upsilon)$.
	Finally, we note that
	\begin{align*}
		\norm{\lambda - \bar\lambda}_{Y\dualspace}
		&=
		4 \varepsilon
		,
		\\
		\norm{g'(\bar x)\adjoint(\lambda - \bar\lambda)}_{X\dualspace}
		&=
		\norm*{
			\parens*{
				\int_{[-1, 1]} 1 \d(\lambda - \bar\lambda)
				,
				\int_{[-1, 1]} s \d(\lambda - \bar\lambda)
			}
		}
		\\&
		=
		\norm*{
			\parens*{
				0,
				\int_{-\varepsilon}^{\varepsilon} s \d s
			}
		}
		=
		0
		.
	\end{align*}
	This shows that \eqref{eq:isol_calmness_ups}
	cannot be satisfied with any constants
	$\delta, \varepsilon, c > 0$.

	This example also has another interesting property.
	For any $y \in \overline K$ in a small neighborhood of $\bar y$, we have
	\begin{equation*}
		\normalcone_K(y)
		=
		\mathcal{M}(\set{s \in [-1,1] \given y(s) = 1 + s^2})_+
		,
	\end{equation*}
	i.e., the normal cone to $K$ at $y$ corresponds to
	the nonnegative measures on the set where $y$ coincides with the upper bound.
	Consequently,
	\begin{equation*}
		\normalcone_K(y)
		+
		\R_- \bar\lambda
		=
		\radialcone_{\normalcone_K(y)}(\bar\lambda)
	\end{equation*}
	is always weak-$\star$ closed.
\end{example}

Let us now come back to the initial question of the section.
To start, we state the following result which is an immediate consequence of
\cite[Proposition~4.47]{BonnansShapiro2000} and \cite{BonnansShapiro2000err}.

\begin{theorem}\label{thm:BonnansShapiro}
	Let $(\bar x,\bar\lambda)\in X\times Y\dualspace$ be a KKT pair of \eqref{eq:problem}
	such that $\Lambda(\bar x)=\set{\bar\lambda}$.
	We further assume that the cones
	\[
		\radialcone_K(g(\bar x)),
		\qquad
		g'(\bar x) X - \radialcone_K(g(\bar x))\cap\bar\lambda\anni
	\]
	are closed while the cone $\radialcone_{\normalcone_K(g(\bar x))}(\bar\lambda)$ is weakly-$\star$ closed.
	Then \eqref{eq:sRZKC} is valid.
\end{theorem}

Let us recall that the uniqueness assumption on the multiplier in \cref{thm:BonnansShapiro}
is trivially satisfied in the presence of \eqref{eq:isol_calmness_ups}.

We note that the closedness assumptions in \cref{thm:BonnansShapiro} 
are difficult to check and rather restrictive if $K$ is nonpolyhedric or $Y$ is infinite dimensional.
Next, we will actually prove that \eqref{eq:isol_calmness_ups}
is sufficient for \eqref{eq:sRZKC}
if $Y$ is finite dimensional and $K$ is polyhedric.

\begin{theorem}\label{thm:IC_gives_sRZKCQ_findim}
	Let $(\bar x,\bar\lambda)\in X\times Y\dualspace$ be a KKT pair of \eqref{eq:problem}
	and assume that \eqref{eq:isol_calmness_ups} is valid.
	Furthermore, assume that $Y$ is finite dimensional and that $K$ is polyhedric.
	Then \eqref{eq:sRZKC} holds.
\end{theorem}
\begin{proof}
	The set
	$Q := g'(\bar x) X - \radialcone_K(g(\bar x))\cap\bar\lambda\anni \subset Y$
	appearing in \eqref{eq:sRZKC} is a convex cone.
	Suppose that $Q \ne Y$.
	Consequently, $0 \in Y$ is a boundary point of $Q$.
	Since $Y$ is finite dimensional,
	\cite[Lemma~4.2.1]{HiriartUrrutyLemarechal2012}
	yields a functional $y\dualspace \in Y\dualspace \setminus \set{0}$
	with $\norm{y\dualspace}_{Y\dualspace} \le \delta/2$
	such that
	\begin{equation*}
		\forall x \in X,\, \forall y \in \overline{ K } \colon
		\qquad
		0
		\le
		\dual{g'(\bar x)\adjoint y\dualspace}{x}_X
		-
		\dual{y\dualspace}{y - g(\bar x)}_Y
		,
	\end{equation*}
	see \eqref{eq:radialcone_to_bar_K} as well.
	Consequently, $g'(\bar x)\adjoint y\dualspace = 0$
	and
	$y\dualspace \in \normalcone_{\overline K}(g(\bar x))$.
	We can apply \cref{lem:fun_with_cones}~\ref{item:fun_with_cones_polyhedric} and
	for each $k\in\N$,
	we find $\lambda_k \in \normalcone_K(g(\bar x))$,
	$\alpha_k \ge 0$, and $\tilde y\dualspace_k \in Y\dualspace$
	such that $\norm{\tilde y\dualspace_k}_{Y\dualspace}\leq \delta/(2k)$
	and
	\[
		y\dualspace = \lambda_k - \alpha_k \bar\lambda + \tilde y\dualspace_k.
	\]
	Since $\bar\lambda \in \normalcone_K(g(\bar x))$, we can assume w.l.o.g.\ that $\alpha_k \ge 1$
	holds for each $k\in\N$.
	Consequently,
	$\lambda_k / \alpha_k \in \normalcone_K(g(\bar x))$ and
	$\norm{ \lambda_k/\alpha_k - \bar\lambda }_{Y\dualspace} = \norm{y\dualspace - \tilde y\dualspace_k}_{Y\dualspace} / \alpha_k \le \delta$.
	Now,
	\eqref{eq:isol_calmness_ups} and $g'(\bar x)\adjoint y\dualspace=0$
	imply
	\begin{equation*}
		\frac1{\alpha_k}
		\norm{ y\dualspace - \tilde y\dualspace_k }_{Y\dualspace}
		=
		\norm{ \lambda_k/\alpha_k - \bar\lambda }_{Y\dualspace}
		\le
		c
		\norm{g'(\bar x)\dualspace ( \lambda_k / \alpha_k - \bar\lambda ) }_{X\dualspace}
		=
		\frac{c}{\alpha_k}
		\norm{ g'(\bar x) \tilde y\dualspace_k }_{X\dualspace}
	\end{equation*}
	and, thus, $\norm{y\dualspace}_{Y\dualspace}-\norm{\tilde y\dualspace_k}_{Y\dualspace}\leq c\norm{g'(\bar x)}_{L(X,Y)}\norm{\tilde y\dualspace_k}_{Y\dualspace}$
	for each large enough $k\in\N$.
	Taking the limit $k\to\infty$ yields a contradiction as $y\dualspace\neq 0$.
\end{proof}

Note that the above proof only utilizes \eqref{eq:isol_calmness_ups}
with $\upsilon = 0$.

The subsequent example indicates that \eqref{eq:isol_calmness_ups},
in general, does not imply \eqref{eq:sRZKC} if $K$ is nonpolyhedric,
even if $Y$ is finite dimensional.
Particularly, the polyhedricity assumption in \cref{thm:IC_gives_sRZKCQ_findim}
cannot be dropped.

\begin{example}\label{ex:no_chance_without_polyhedricity}
	Let us investigate \eqref{eq:problem},
	where
	\[
		\begin{aligned}
		X&=\R,&	\quad Y&=\R^2&\quad &&
		\\
		f(x)&= -x,&\quad g(x)&=(x^2,x+1),&\quad
		K&=\set{y\in\R^2 \given \norm{y}\leq 1}.&
		\end{aligned}
	\]
	As the feasible set of this problem equals $[-1,0]$,
	its uniquely determined minimizer is $\bar x=0$.
	Throughout, we use $\bar y=g(\bar x)=(0,1)$.
	
	As we have 
	\[	
		\radialcone_K(\bar y)=\set{d\in\R^2 \given d_2<0}\cup\set{(0,0)},
		\qquad
		\normalcone_K(\bar y)=\set{0}\times\R_+,
	\]
	one can easily check that the KKT conditions
	yield $\Lambda(\bar x)=\set{(0,1)}$.
	However, \eqref{eq:sRZKC} with $\bar\lambda = (0,1)$ fails
	since $g'(\bar x)X=\set{0}\times\R$ and $\radialcone_K(\bar y)\cap\bar\lambda\anni=\set{(0,0)}$.
	
	Let us consider a pair $(\xi,\upsilon) \in B_{1/2}^{\R\times\R^2}((0,0))$
	and some $\lambda\in\Upsilon_{\bar x}(\xi,\upsilon)$.
	The definition of $\Upsilon_{\bar x}$ yields
	\begin{equation}\label{eq:ex_nonpolyhedric_pertubed_KKT_system}
		-1 + \lambda_2 = \xi,
		\qquad
		\lambda\in\normalcone_K(\bar y +\upsilon).
	\end{equation}
	If $\norm{\bar y + \upsilon}<1$, then $\normalcone_K(\bar y + \upsilon)=\set{(0,0)}$
	and the above requires $\xi=-1$ which is not possible
	in the considered ball.
	If $\norm{\bar y+\upsilon}>1$, then $\normalcone_K(\bar y+\upsilon)=\emptyset$.
	Hence, it remains to consider the case $\norm{\bar y + \upsilon}=1$.
	Then $\normalcone_K(\bar y+\upsilon)=\cone(\set{\bar y+\upsilon})$,
	and \eqref{eq:ex_nonpolyhedric_pertubed_KKT_system} reduces to 
	\[
		-1 + \lambda_2 = \xi,
		\qquad
		\lambda_1 = \alpha\,\upsilon_1,
		\qquad
		\lambda_2 = \alpha\,(1+\upsilon_2),
		\qquad
		\alpha\geq 0.
	\]	
	In the considered ball, this yields $\alpha = (1+\xi)/(1+\upsilon_2)$, 
	and
	\[
		\lambda = (1+\xi)\left(\frac{\upsilon_1}{1+\upsilon_2},1\right).
	\]
	Hence, on the considered ball, $\Upsilon_{\bar x}$ is a single-valued
	locally Lipschitz continuous function,
	and the latter is clearly sufficient for \eqref{eq:isol_calmness_ups}.
\end{example}

Our final example in this section shows that,
even in the setting of reflexive Banach spaces $X$ and $Y$ 
and a polyhedric set $K$,
it may happen that \eqref{eq:isol_calmness_ups} is valid
while \eqref{eq:sRZKC} is violated.
Note that this example is a minor modification of \cref{ex:weak_sRZKC_vs_RZKCQ}.

\begin{example}
	\label{ex:yet_another_counterexample}
	Let us consider the setting
	\begin{align*}
		X &= L^2(0,1),
		&
		Y &= L^2(0,1) \times L^2(0,1),
		&
		K &= L^2(0,1)_+ \times L^2(0,1)_+,
		\\
		\bar y & = g(\bar x) = (\zerofunction, \onefunction),
		&
		f'(\bar x) &= \onefunction,
		&
		g'(\bar x) x &= (x,x)
		.
	\end{align*}
	The precise value of $\bar x$ is not important.
	From \cref{ex:weak_sRZKC_vs_RZKCQ}
	we get the formulas \eqref{eq:some_cones_in_L2},
	and we still have $\bar\lambda=(-\onefunction,\zerofunction)\in\Lambda(\bar x)$.
	Hence, we still may rely on the formulas from \eqref{eq:some_annihilated_cones_in_L2}.
	
	One the one hand,
	for an arbitrary pair $(y_1, y_2) \in Y$, we have
	\begin{align*}
		y_1
		&=
		\max(y_1, y_2) - \max(\zerofunction, y_2 - y_1)
		,
		\\
		y_2
		&=
		\max(y_1, y_2) - \max(\zerofunction, y_1 - y_2)
		,
	\end{align*}
	where $\max$ has to be understood in pointwise fashion.
	Noting that $\max(y_1,y_2)\in L^2(0,1)$ and 
	$\max(\zerofunction,y_2-y_1),\max(\zerofunction,y_1-y_2)\in L^2(0,1)_+$,
	this shows
	\begin{equation*}
		Y
		= 
		g'(\bar x) X - \radialcone_K(\bar y)
		,
	\end{equation*}
	i.e., \eqref{eq:RZKCQ} is valid.
	Furthermore, validity of
	\begin{equation*}
		Y = g'(\bar x) X - \tangentcone_K(\bar y) \cap \bar\lambda\anni
	\end{equation*}
	is obvious.
	On the other hand, we can choose
	$y_2 \in L^2(0,1) \setminus \parens*{ L^2(0,1)_+ + L^\infty(0,1) }$.
	Then it is easy to check that
	$(\zerofunction, y_2)$ does not belong to $g'(\bar x) X - \radialcone_K(\bar y) \cap \bar\lambda\anni$.
	Consequently,
	\begin{equation*}
		Y
		\ne
		g'(\bar x) X - \radialcone_K(\bar y) \cap \bar\lambda\anni
		,
	\end{equation*}
	i.e., \eqref{eq:sRZKC} is violated. 

	Finally,
	we check that \eqref{eq:isol_calmness_ups} holds.
	To this end, let
	$\upsilon \in Y$ and
	$\lambda \in \normalcone_K(\bar y + \upsilon)$
	be arbitrary.
	Since $\bar y_1 = \zerofunction$, we have $\upsilon_1 \ge \zerofunction$.
	W.l.o.g., we can further assume that $\upsilon_1 = \zerofunction$,
	since the normal cone would become smaller for $\upsilon_1$ being positive 
	on a subset of $(0,1)$ of positive measure.
	Similarly, $\upsilon_2 \ge -1$.
	Next, we define the set
	$I = \set{ t \in (0,1) \given \upsilon_2(t) = -1}$.
	Note that
	$\lambda_2$ vanishes outside of $I$, i.e.,
	we may identify 
	$\lambda_2 \in L^2(I)_-$.
	Moreover,
	$\norm{\upsilon_2}_{L^2(0,1)} \ge \abs{I}^{1/2}$
	holds due to $\upsilon_2\equiv -1$ on $I$.

	Thus, validity of \eqref{eq:isol_calmness_ups}
	follows if we are in a position to show
	\begin{align*}
		\norm{\lambda - \bar\lambda}_{Y\dualspace}
		&=
		\parens*{
			\norm{\lambda_1 + \onefunction}_{L^2(0,1)}^2
			+
			\norm{\lambda_2}_{L^2(I)}^2
		}^{1/2}
		\\&
		\le
		c
		\parens*{
			\norm{ g'(\bar x)\adjoint (\lambda_1 + \onefunction, \lambda_2)}_{L^2(0,1)}
			+
			\abs{I}^{1/2}
		}
		\\&
		=
		c
		\parens*{
			\norm{ \lambda_1 + \onefunction + \lambda_2}_{L^2(0,1)}
			+
			\abs{I}^{1/2}
		}
	\end{align*}
	for $\lambda_1 \in L^2(0,1)_-$ and $\lambda_2 \in L^2(I)_-$.
	Note that this inequality is implied by the inequality
	\begin{equation*}
		\norm{\lambda_1 + \onefunction}_{L^2(0,1)}^2
		+
		\norm{\lambda_2}_{L^2(I)}^2
		\le
		\tilde c
		\parens*{
			\norm{ \lambda_1 + \onefunction + \lambda_2}_{L^2(0,1)}^2
			+
			\abs{I}
		}
	\end{equation*}
	for $\lambda_1 \in L^2(0,1)_-$ and $\lambda_2 \in L^2(I)_-$
	where $\tilde c:=c^2$.
	Since all the terms can be written as integrals,
	it remains to check 
	\begin{align*}
		(1-\alpha)^2 + 0
		&\le \tilde c \parens*{ (1-\alpha-0)^2 + 0 }
		,
		\\
		(1-\alpha)^2 + \beta^2
		&\le \tilde c \parens*{ (1-\alpha-\beta)^2 + 1 }
	\end{align*}
	for all $\alpha,\beta\in\R_+$,
	where the first inequality
	corresponds to points $t \in (0,1) \setminus I$
	and the second one to $t \in I$.
	The first inequality holds trivially for all $\tilde c \ge 1$
	and the second inequality holds, by \cref{lem:some_inequality}, for $\tilde c=3$.
	Consequently, \eqref{eq:isol_calmness_ups}
	is satisfied with $\delta = \varepsilon = \infty$
	and $c = \sqrt 3$.

	To summarize,
	in this example,
	\eqref{eq:RZKCQ}, \eqref{eq:isol_calmness_ups}, \eqref{eq:weak_sRZKC_i}, and \eqref{eq:weak_sRZKC_ii}
	hold, see \cref{thm:yet_another_theorem} as well, but \eqref{eq:sRZKC} is violated.
	Moreover, it has most of the nice properties
	that one could hope for,
	e.g., reflexivity of $X$ and $Y$ and polyhedricity of $K$, see \cite[Theorem~3.58]{BonnansShapiro2000}.
	However, $Y$ is infinite dimensional.

	It is also instructive to modify the example slightly
	by using $\bar y_2$ defined via $\bar y_2(s) = s$
	and keeping all other data unchanged.
	Note that after this modification of $\bar y_2$,
	we still have $\bar\lambda\in\Lambda(\bar x)$,
	\begin{equation*}
		Y
		= g'(\bar x) X - \radialcone_K(\bar y)
		= g'(\bar x) X - \tangentcone_K(\bar y) \cap \bar\lambda\anni,
	\end{equation*}
	and
	\begin{equation*}
		Y
		\ne
		g'(\bar x) X - \radialcone_K(\bar y) \cap \bar\lambda\anni
	\end{equation*}
	by the very same arguments.
	However, \eqref{eq:isol_calmness_ups} is now violated.
	Indeed, for $\varepsilon > 0$,
	we consider
	\begin{equation*}
		\upsilon = (\zerofunction, -\chi_{(0,\varepsilon)} \bar y_2),
		\qquad
		\lambda = (-\onefunction + \chi_{(0,\varepsilon)}, -\chi_{(0,\varepsilon)}).
	\end{equation*}
	It is clear that $\lambda \in \normalcone_K(\bar y + \upsilon)$.
	However, a direct calculation shows
	\begin{align*}
		\norm{\lambda - \bar\lambda}_{Y}
		&=
		\parens*{
			\norm{\chi_{(0,\varepsilon)}}_{L^2(0,1)}^2
			+
			\norm{\chi_{(0,\varepsilon)}}_{L^2(0,1)}^2
		}^{1/2}
		=
		\sqrt{2 \varepsilon},
		\\
		\norm{g'(\bar x)\adjoint ( \lambda - \bar\lambda)}_{X\dualspace} + \norm{\upsilon}_Y
		&=
		\norm{\lambda_1 + \onefunction + \lambda_2}_{L^2(0,1)} + \norm{\upsilon_2}_{L^2(0,1)}
		=
		\frac{\varepsilon^{3/2}}{\sqrt{3}}
		.
	\end{align*}
	Consequently, \eqref{eq:isol_calmness_ups} cannot hold
	for any constants $c, \varepsilon, \delta > 0$.

	The same effect can be observed when choosing the initial $\bar y = (\zerofunction,\onefunction)$
	but taking $\bar\lambda_1(s) = -s^{-1/4}$ and $\bar\lambda_2=\zerofunction$
	while adjusting $f'(\bar x)=-\bar\lambda_1$.
	Then one can consider the perturbation
	\begin{equation*}
		\upsilon = (\zerofunction, -\chi_{(0,\varepsilon)} \bar y_2),
		\qquad
		\lambda = (\chi_{(\varepsilon,1)} \bar\lambda_1, \chi_{(0,\varepsilon)}\bar\lambda_1).
	\end{equation*}
	Then, some calculations show
	\begin{align*}
		\norm{\lambda - \bar\lambda}_{Y}
		&=
		\parens*{2 \int_0^\varepsilon s^{-1/2} \d s }^{1/2}
		=
		2 \varepsilon^{1/4}
		,
		\\
		\norm{g'(\bar x)\adjoint ( \lambda - \bar\lambda)}_{X\dualspace} + \norm{\upsilon}_Y
		&=
		\norm{\lambda_1-\bar\lambda_1+\lambda_2}_{L^2(0,1)}
		+
		\norm{\upsilon_2}_{L^2(0,1)}
		=
		\varepsilon^{1/2}
		.
	\end{align*}
	Again, \eqref{eq:isol_calmness_ups} cannot hold for any constants $c, \varepsilon, \delta > 0$.
\end{example}

Taking our considerations from \cref{sec:QCs} and \cref{thm:IC_gives_sRZKCQ_findim}
together indicates validity of the following result.

\begin{corollary}\label{cor:all_the_same}
	Let $(\bar x,\bar\lambda)\in X\times Y\dualspace$ be a KKT pair of \eqref{eq:problem}.
	Furthermore, assume that $Y$ is finite dimensional and that $K$ is polyhedric.
	Then conditions \eqref{eq:sRZKC}, \eqref{eq:isol_calmness_ups}, \eqref{eq:weak_sRZKC_i},
	and \eqref{eq:weak_sRZKC_ii} are all equivalent.
\end{corollary}

In \cref{fig:summary},
we summarize the essentially well known observations from \cref{sec:QCs}
and our new findings from this section in order to clarify
the relations between \eqref{eq:sRZKC}, \eqref{eq:RZKCQ},
\eqref{eq:isol_calmness_ups}, \eqref{eq:weak_sRZKC_i}, and \eqref{eq:weak_sRZKC_ii}
in visually appealing way.

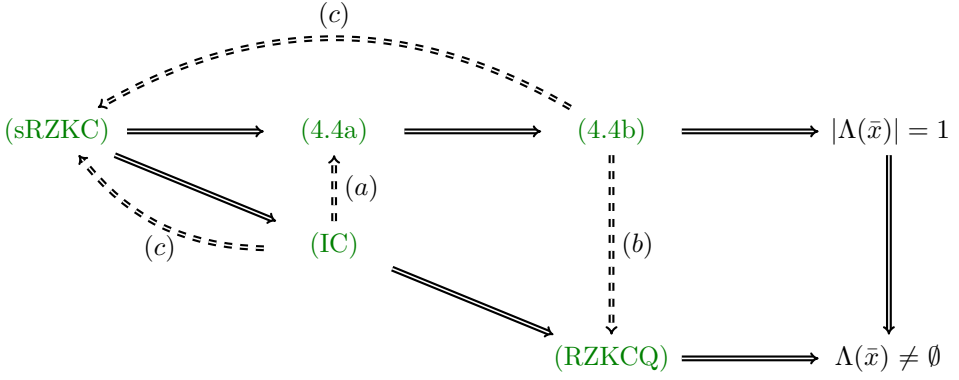
\begin{figure}[ht]
\centering
\begin{tikzpicture}[->]

  \node[punkt] at (0,0) 	(A){\eqref{eq:sRZKC}};
  \node[punkt] at (11/3,0) 	(B){\eqref{eq:weak_sRZKC_i}};
  \node[punkt] at (11*2/3,0)		(C){\eqref{eq:weak_sRZKC_ii}};
  \node[punkt] at (11/3,-1.5)	(D){\eqref{eq:isol_calmness_ups}};
  \node[punkt] at (11*2/3,-3)	(E){\eqref{eq:RZKCQ}};
  \node[punkt] at (11,0)	(F){$\abs{\Lambda(\bar x)}=1$};
  \node[punkt] at (11,-3)	(G){$\Lambda(\bar x)\neq\emptyset$};

  \path     (A) edge[-implies,thick,double] node {}(B)
  			(A) edge[-implies,thick,double] node {}(D)
  			(B) edge[-implies,thick,double] node {}(C)
  			(C) edge[-implies,thick,double,dashed,bend right] node[above] {$(c)$}(A)
  			(C) edge[-implies,thick,double,dashed] node[right] {$(b)$}(E)
  			(C) edge[-implies,thick,double] node {}(F)
  			(D) edge[-implies,thick,double,dashed,bend angle=25,bend left] node[below] {$(c)$}(A)
  			(D) edge[-implies,thick,double,dashed] node[right] {$(a)$}(B)
  			(D) edge[-implies,thick,double] node {}(E)
  			(E) edge[-implies,thick,double] node {}(G)
  			(F) edge[-implies,thick,double] node {}(G)
  			;
\end{tikzpicture}
\caption{%
	Relations between qualification conditions for some feasible point $\bar x\in X$ of \eqref{eq:problem}.
	Dashed relations hold in the presence of additional assumptions:
	(a) requires condition \ref{item:Yref}, \ref{item:radial_cone_weak_star_closed}, or \ref{item:range_finite_dimensional}
		from \cref{thm:yet_another_theorem},
	(b) requires $Y$ to be finite dimensional,
	(c) requires $Y$ to be finite dimensional and $K$ to be polyhedric.
}%
\label{fig:summary}
\end{figure}

In \cref{tab:ex_overview}, we summarize the features of all examples stated in this paper.
The fact that \eqref{eq:isol_calmness_ups} fails in \cref{ex:weak_sRZKC_vs_RZKCQ}
is a consequence of \cref{thm:IC_gives_RZKCQ}.
Polyhedricity of the sets $K$ used in \cref{ex:weak_sRZKC_vs_RZKCQ,ex:this_condition_is_not_enough,ex:yet_another_counterexample}
follows from \cite[Theorem~3.58]{BonnansShapiro2000}
(for \cref{ex:weak_sRZKC_vs_RZKCQ,ex:yet_another_counterexample})
and \cite[Theorem~4.18]{Wachsmuth2019}
(for \cref{ex:this_condition_is_not_enough}).

\begin{table}[ht]
	\centering
	\begin{tabular}{cccccc}
	\toprule
	Example	&	\eqref{eq:isol_calmness_ups} & \eqref{eq:weak_sRZKC}
		& \eqref{eq:RZKCQ}	&	$\dim(Y)<\infty$	&	$K$ polyhedric
	\\
	\midrule
	\cref{ex:weak_sRZKC_vs_RZKCQ}
		& \xmark	& \cmark	& \xmark	& \xmark	& \cmark
	\\
	\cref{ex:multiplier_unique_but_nothing_else}
		& \xmark	& \xmark	& \cmark	& \cmark	& \xmark
	\\
	\cref{ex:this_condition_is_not_enough}
		& \xmark	& \cmark	& \cmark	& \xmark	& \cmark
	\\
	\cref{ex:no_chance_without_polyhedricity}
		& \cmark	& \cmark	& \cmark	& \cmark	& \xmark
	\\
	\cref{ex:yet_another_counterexample}
		& \cmark	& \cmark	& \cmark	& \xmark	& \cmark
	\\
	\bottomrule
	\end{tabular}
	\caption{%
		Features of the examples in this paper.
		In all of them, \eqref{eq:sRZKC} is violated
		and the multiplier under consideration is uniquely determined.
		Furthermore, in all examples, both conditions
		\eqref{eq:weak_sRZKC_i} and \eqref{eq:weak_sRZKC_ii} in \eqref{eq:weak_sRZKC}
		either hold in parallel or are violated at the same time,
		so that we use just one column for their documentation.
	}%
	\label{tab:ex_overview}
\end{table}

Let us close this section by a remark which is concerned with a
seemingly stronger version of condition \eqref{eq:isol_calmness_ups}.

\begin{remark}\label{rem:IC_delta_infty}
	Let $(\bar x,\bar\lambda)\in X\times Y\dualspace$ be a KKT pair of \eqref{eq:problem}.
	According to \cite{BonnansShapiro2000}, validity of \eqref{eq:sRZKC} already implies
	the existence of $\varepsilon,c>0$ such that
	\begin{equation}\label{eq:strong_IC}
	\begin{aligned}
	\forall (\xi,\upsilon)\in B_\varepsilon^{X\dualspace\times Y}((0,0)),\,
	&\forall \lambda\in\Upsilon_{\bar x}(\xi,\upsilon)
	\colon\quad
	\norm{\lambda-\bar \lambda}_{Y\dualspace}
	\leq
	c(\norm{\xi}_{X\dualspace}+\norm{\upsilon}_{Y})
	\end{aligned}
	\end{equation}
	is valid, i.e., \eqref{eq:isol_calmness_ups} holds with $\delta:=\infty$.
	Clearly, \eqref{eq:strong_IC} is sufficient for \eqref{eq:isol_calmness_ups},
	and, following \cref{cor:all_the_same}, the converse is true if $Y$ is finite dimensional and $K$ is polyhedric.
	However, it is not clear whether \eqref{eq:strong_IC} and \eqref{eq:isol_calmness_ups} are
	equivalent in the absence of additional assumptions,
	and addressing this open problem might be a promising subject of future research.
	
	We can give an answer for the situation where the constraints in 
	\eqref{eq:pertubed_problem} are not perturbed, i.e., if $\upsilon=0$ is fixed.
	Here, \eqref{eq:strong_IC} and \eqref{eq:isol_calmness_ups} are, indeed, equivalent.
	In order to verify this, let us assume that \eqref{eq:isol_calmness_ups} holds for $\varepsilon,\delta,c>0$,
	and pick $\lambda\in\Upsilon_{\bar x}(\xi,0)$ where $\xi\in B_\varepsilon^{X\dualspace}(0)$ is valid.
	For each $\alpha\in[0,1]$, $\hat\lambda_\alpha:=\alpha\lambda+(1-\alpha)\bar\lambda$
	belongs to $\normalcone_K(g(\bar x))$ as the latter set is convex,
	and, for small enough $\alpha>0$, we even have $\hat\lambda_\alpha\in B_\delta^{Y\dualspace}(\bar\lambda)$.
	Furthermore, 
	\[
		\LL'_x(\bar x,\hat\lambda_\alpha)
		=
		\alpha\,g'(\bar x)\adjoint(\lambda-\bar\lambda)
		=
		\alpha\,\LL'_x(\bar x,\lambda)
		=
		\alpha\xi
	\]
	is valid,
	so that \eqref{eq:isol_calmness_ups} implies
	\[
		\norm{\hat\lambda_\alpha-\bar \lambda}_{Y\dualspace}\leq c\norm{\alpha\xi}_{X\dualspace}.
	\]
	Due to $\hat\lambda_\alpha-\bar\lambda=\alpha(\lambda-\bar\lambda)$,
	this also shows
	\[
		\norm{\lambda-\bar\lambda}_{Y\dualspace}\leq c\norm{\xi}_{X\dualspace},
	\]
	i.e., \eqref{eq:strong_IC} restricted to $\upsilon = 0$ follows.
\end{remark}

\section{Conclusions}\label{sec:conclusions}

In this paper, we have shown that condition \eqref{eq:isol_calmness_ups},
the isolated calmness of the restricted Lagrange multiplier mapping $\Upsilon_{\bar x}$ at $((0,0),\bar\lambda)$, 
is always sufficient for the validity of \eqref{eq:RZKCQ} and,
whenever $Y$ is finite dimensional while $K$ is polyhedric, for the validity of \eqref{eq:sRZKC},
see \cref{thm:IC_gives_RZKCQ,thm:IC_gives_sRZKCQ_findim}.
Some illustrative examples have been presented to underline that these additional assumptions, indeed,
cannot be omitted in general.
As mentioned in \cref{rem:IC_delta_infty}, it would be an interesting question of future research
to clarify whether \eqref{eq:isol_calmness_ups} already implies the stronger condition \eqref{eq:strong_IC}.
Furthermore, one may study isolated calmness type properties of the more general Lagrange multiplier mapping
$\Upsilon\colon X\dualspace\times Y\rightrightarrows X\times Y\dualspace$ given by
\[
	\forall \xi\in X\dualspace,\,\forall \upsilon\in Y\colon\quad
	\Upsilon(\xi,\upsilon)
	:=
	\set{(x,\lambda)\in X\times Y\dualspace \given \LL'_x(x,\lambda)=\xi,\,\lambda\in\normalcone_K(g(x)+\upsilon)}
\]
which, naturally, should be related to second-order conditions associated with \eqref{eq:problem},
see, e.g., \cite[Theorem~4.51]{BonnansShapiro2000}.

\appendix
\section{A helpful inequality}

The following lemma provides an estimate which is used in \cref{ex:yet_another_counterexample}.

\begin{lemma}\label{lem:some_inequality}
	We have
	\[
		\forall \alpha,\beta\in\R_+\colon\quad
		(1-\alpha)^2+\beta^2
		\leq 
		3\parens*{
			(1-\alpha-\beta)^2+1
		}.
	\]
\end{lemma}
\begin{proof}
	Given $\alpha,\beta\in\R_+$, 
	the desired inequality is, obviously, equivalent to
	\[
		2(1-\alpha)\beta\leq 2(1-\alpha-\beta)^2+3
	\]
	and, thus, to
	\begin{equation}\label{eq:some_equivalent_inequality}
		3(1-\alpha)\beta\leq (1-\alpha)^2 + \beta^2 + \frac32.
	\end{equation}
	We note that the latter is trivial if $\alpha\geq 1$
	as the left-hand side is nonpositive in this case.
	
	Hence, let us fix some arbitrary $\alpha\in[0,1)$ and consider the convex parabola 
	\[
		\beta\mapsto \beta^2-3(1-\alpha)\beta+(1-\alpha)^2+\frac32.
	\]
	The latter possesses the apex $(\frac32(1-\alpha),\frac32-\frac54(1-\alpha)^2)$,
	and as $\frac32-\frac54(1-\alpha)^2\geq\frac14>0$,
	\eqref{eq:some_equivalent_inequality} holds for $\alpha\in[0,1)$ as well.
\end{proof}

%%fakesection: Bibliography
%\bibliographystyle{siamplain}
%\bibliography{references}

\end{document}